\documentclass[twoside,11pt]{article}
\usepackage{jmlr2earxiv}

\usepackage{amsmath} 
\usepackage{amsthm}
\usepackage{mathtools}
\usepackage{bbm}
\usepackage{latexsym}
\usepackage{dsfont}
\usepackage{pifont}
\usepackage{color} 
\usepackage{tikz}
\usepackage{verbatim}
\usepackage{hyperref}
\usepackage{tikz}
\usetikzlibrary{arrows}
\usetikzlibrary{calc}
\usetikzlibrary{positioning}
\usepackage{times}

\graphicspath{{../Figures/}}

\newtheorem{theorem}{Theorem}
\newtheorem{lemma}[theorem]{Lemma}

\newtheorem{proposition}[theorem]{Proposition}
\theoremstyle{definition}
\newtheorem{example}[theorem]{Example}
\newtheorem{definition}[theorem]{Definition}
\newtheorem{remark}[theorem]{Remark}

\newcommand{\Rcal}{\mathcal{R}}
\newcommand{\argmax}{\operatorname{argmax}}
\newcommand{\supp}{\operatorname{supp}}
\newcommand{\Pcal}{\mathcal{P}}
\newcommand{\Gcal}{\mathcal{G}}
\newcommand{\Fcal}{\mathcal{F}}
\newcommand{\Wcal}{\mathcal{W}}

\newcommand{\sym}{\operatorname{sym}}
\newcommand{\conv}{\operatorname{conv}}
\definecolor{bl}{RGB}{20,20,150}

\ShortHeadings{Geometry and Determinism of POMDPs}{Mont\'ufar, Ay, Ghazi-Zahedi}
\firstpageno{1}

\begin{document}
\thispagestyle{empty}
\title{Geometry and Determinism of Optimal Stationary Control in Partially Observable Markov Decision Processes}

\author{\name Guido Mont\'ufar \email montufar@mis.mpg.de \\
	\addr Max Planck Institute for Mathematics in the Sciences\\
	04103 Leipzig, Germany
	\AND
	\name Keyan Ghazi-Zahedi \email zahedi@mis.mpg.de \\
	\addr Max Planck Institute for Mathematics in the Sciences\\
	04103 Leipzig, Germany
	\AND
	\name Nihat Ay \email nay@mis.mpg.de \\
	\addr Max Planck Institute for Mathematics in the Sciences\\
	04103 Leipzig, Germany \smallskip \\
	Faculty of Mathematics and Computer Science \\
	Leipzig University \\
	04009 Leipzig, Germany \smallskip \\
	Santa Fe Institute \\
	Santa Fe, NM 87501, USA
	}
\editor{}

\maketitle

\begin{abstract}
It is well known that for any finite state Markov decision process (MDP) there is a memoryless deterministic policy that maximizes the expected reward. For partially observable Markov decision processes (POMDPs), optimal memoryless policies are generally stochastic. We study the expected reward optimization problem over the set of memoryless stochastic policies. We formulate this as a constrained linear optimization problem and develop a corresponding geometric framework. We show that any POMDP has an optimal memoryless policy of limited stochasticity, which allows us to reduce the dimensionality of the search space. Experiments demonstrate that this approach enables better and faster convergence of the policy gradient on the evaluated systems. 
\end{abstract}

\begin{keywords}
MDP, POMDP, partial observability, memoryless stochastic policy, average reward, policy gradient, reinforcement learning
\end{keywords}

\section{Introduction}

The field of reinforcement learning addresses a broad class of problems where an agent has to learn how to act in order to maximize some form of cumulative reward. 
On choosing action $a$ at some world state $w$ the world undergoes a transition to state $w'$ with probability $\alpha(w'|w,a)$ and the agent receives a reward signal $R(w,a,w')$. 
A policy is a rule for selecting actions based on the information that is available to the agent at each time step. 
In the simplest case, the Markov decision process (MDP), the full world state is available to the agent at each time step. 
A key result in this context shows the existence of optimal policies which are memoryless and deterministic~\citep[see][]{Ross:1983:ISD:538843}. 
In other words, the agent performs optimally by choosing one specific action at each time step based on the current world state. 
The agent does not need to take the history of world states into account, nor does he need to randomize his actions. 

In many cases one has to assume that the agent experiences the world only through noisy sensors and the agent has to choose actions based only the partial information provided by these sensors. 
More precisely, if the world state is $w$, the agent only observes a sensor state $s$ with probability $\beta(s|w)$. 
This setting is known as partially observable Markov decision process (POMDP). 
Policy optimization for POMDPs has been discussed by several authors~\citep[see][]{Sondik1978,Chrisman1992aReinforcement, Littman1995aLearning,Mccallum1996aReinforcement,Parr1995aApproximating}. 
Optimal policies generally need to take the history of sensor states into account. 
This requires that the agent be equipped with a memory that stores the sensor history or an encoding thereof (e.g., a properly updated belief state) which may require additional computation. 

Although in principle possible, in practice it is often too expensive to find or even to store and execute completely general optimal policies. 
Some form of representation or approximation is needed. 
In particular, in the context of embodied artificial intelligence and systems design~\citep{Pfeifer2006aHow-the-Body} the on-board computation sets limits to the complexity of the controller with respect to both, memory and computational cost. 
We are interested in policies with limited memory~\citep[see, e.g.,][]{Hansen:1998:FCP:928126}. 
In fact we will focus on memoryless stochastic  policies~\citep[see][]{singh1994learning,Jaakkola95reinforcementlearning}. 
Memoryless policies may be worse than policies with memory, but they require far fewer parameters and computation. 
Among other approaches, the GPOMDP algorithm~\citep{Baxter:2001:IPE:1622845.1622855} provides a gradient based method to optimize the expected reward over parametric models of memoryless stochastic policies. 
For interesting systems, the set of all memoryless stochastic policies can still be very high dimensional and it is important to find good models. 
In this article we show that each POMDP has an optimal memoryless policy of limited stochasticity, which allows us to construct low-dimensional differentiable policy models with optimality guarantees. 
The amount of stochasticity can be bounded in terms of the amount of perceptual aliasing, independently of the specific form of the reward signal. 

We follow a geometric approach to memoryless policy optimization for POMDPs. 
The key idea is that the objective function (the expected reward per time step) can be regarded as a linear function over the set of stationary joint distributions over world states and actions. 
For MDPs this set is a convex polytope and, in turn, there always exists an optimizer which is an extreme point. 
The extreme points correspond to deterministic policies (which cannot be written as convex combinations of other policies). 
For POMDPs this set is in general not convex, but it can be decomposed into convex pieces. 
There exists an optimizer which is an extreme point of one of these pieces. 
Depending on the dimension of the convex pieces, the optimizer is more or less stochastic. 

This paper is organized as follows. 
In Section~\ref{sec:POMDP} we review basics on POMDPs. 
In Section~\ref{sec:optimalcontrol} we discuss the reward optimization problem in POMDPs as a constrained linear optimization problem with two types of constraints. 
The first constraint is about the types of policies that can be represented in the underlying MDP. 
The second constraint relates policies with stationary world state distributions. 
We discuss the details of these constraints in Sections~\ref{sec:representability} and~\ref{sec:stationarity}. 
In Section~\ref{sec:determinism} we use these geometric descriptions to show that any POMDP has an optimal stationary policy of limited stochasticity. 
In Section~\ref{sec:models} we apply the stochasticity bound to define low dimensional policy models with optimality guarantees. 
In Section~\ref{sec:experiments} we present experiments which demonstrate the usefulness of the proposed models. 
In Section~\ref{sec:conclusion} we offer our conclusions.

\section{Partially observable Markov decision processes}
\label{sec:POMDP}

A discrete time partially observable Markov decision process (POMDP) is defined by a tuple $(W,S,A,\alpha,\beta, R)$, 
where $W$ is a finite set of world states, 
$S$ is a finite set of sensor states, 
$A$ is a finite set of actions, 
$\beta\colon W\to \Delta_S$ is a Markov kernel that describes sensor state probabilities given the world state, 
$\alpha\colon W\times A\to \Delta_W$ is a Markov kernel that describes the probability of transitioning to a world state given the current world state and action, 
$R\colon W\times A\to \mathbb{R}$ is a reward signal. 
A Markov decision process (MDP) is the special case where $W=S$ and $\beta$ is the identity map. 

A policy $\pi$ is a mechanism for selecting actions. 
In general, at each time step $t\in\mathbb{N}$, a policy is defined by a Markov kernel $\pi_t$ taking the history $h_t= (s_0,a_0,\ldots, s_t)$ of sensor states and actions to a probability distribution $\pi_t(\cdot|h_t)$ over $A$. 
A policy is \emph{deterministic} when at each time step each possible history leads to a single positive probability action. 
A policy is \emph{memoryless} when the distribution over actions only depends on the current sensor state, $\pi_t(\cdot|h_t)=\pi_t(\cdot|s_t)$. 
A policy is \emph{stationary} (homogeneous) when it is memoryless and time independent, $\pi_t(\cdot|h_t)=\pi(\cdot|s_t)$ for all~$t$. 
Stationary policies are represented by kernels of the form $\pi\colon S\to \Delta_A$. We denote the set of all such policies by $\Delta_{S,A}$. 

The goal is to find a policy that maximizes some form of expected reward. 
We consider the long term expected reward per time step (also called average reward)
\begin{equation}
\Rcal_{\mu}(\pi) = \lim_{T\to \infty}
\mathbb{E}_{\Pr\left\{ (w_t,a_t)_{t=0}^{T-1}\middle|\pi, \mu \right\}}\left[\frac{1}{T}\sum_{t=0}^{T-1} R(w_t, a_t) \right].  
\label{eq:pertimeexpectedreward}
\end{equation}
Here $\Pr\left\{(w_t,a_t)_{t=0}^{T-1}\middle|\pi, \mu \right\}$ is the probability of the sequence $w_0,a_0, w_1, a_1, \ldots, w_{T-1},a_{T-1}$, 
given that $w_0$ is distributed according to the start distribution $\mu\in\Delta_W$ and at each time step actions are selected according to the policy $\pi$. 
Another option is to consider a discount factor $\gamma\in(0,1)$ and the discounted long term expected reward 
\begin{equation}
\Rcal^\gamma_{\mu}(\pi) = \lim_{T\to\infty} \mathbb{E}_{\Pr\left\{ (w_t,a_t)_{t=0}^{T-1}\middle|\pi, \mu \right\}}\left[\sum_{t=0}^{T-1} \gamma^t R(w_t, a_t) \right]. 
\label{eq:discountedexpectedreward}
\end{equation} 

In the case of an MDP, it is always possible to find an optimal memoryless deterministic policy. 
In other words, there is a policy that chooses an action deterministically at each time step, depending only on the current world state, which achieves the same or higher long term expected reward as any other policy. 
This fact can be regarded as a consequence of the policy improvement theorem~\citep{Bellman:1957,howard1960dynamic}. 
% The policy improvement theorem states that a greedy policy $\pi'$ of any given value function $V^\pi(w)$ produces at least the same value in each state, $V^{\pi'}\geq V^{\pi}$~\citep[see][]{suttonbarto98}. 
% The state value function $V^\pi(s)$ of a policy $\pi$ is the unique solution of a linear system of equations (Bellman's equation). 
% Namely (I - \gamma (\pi(|s)\alpha(|s,))_s ) V = (\pi(|s)R(s,))_s
% The matrix (I - \gamma (\pi(|s)\alpha(|s,))_s ) is invertible for if it wasnt, then 1 would be an eigenvalue of \gamma (\pi(|s)\alpha(|s,))_s, which cant be, since (\pi(|s)\alpha(|s,))_s is a stochastic matrix its eigenvalues are at most one and \gamma < 1. 

In the case of a POMDP, policies with memory may perform much better than the memoryless policies. 
Furthermore, within the memoryless policies, stochastic policies may perform much better than the deterministic ones~\citep[see][]{singh1994learning}. 
The intuitive reason is simple: Several world states may produce the same sensor state with positive probability (perceptual aliasing). 
On the basis of such a sensor state alone, the agent cannot discriminate the underlying world state with certainty. 
On different world states the same action may lead to drastically different outcomes. 
Sometimes the agent is forced to choose probabilistically between the optimal actions for the possibly underlying world states (see Example~\ref{sec:example}). 
Sometimes he is forced to choose suboptimal actions in order to minimize the risk of catastrophic outcomes (see Example~\ref{fig:example}). 
On the other hand, the sequence of previous sensor states may help the agent identify the current world state and choose one single optimal action. 
This illustrates why in POMDPs optimal policies may need to take the entire history of sensor states into account and also why the optimal memoryless policies may require stochastic action choices. 

The set of policies that take the histories of sensor states and actions into account grows extremely fast. 
A common approach is to transform the POMDP into a belief-state MDP, where the discrete sensor state is replaced by a continuous Bayesian belief about the current world state. 
Such belief states encode the history of sensor states and allow for representations of optimal policies. 
However, belief states are associated with costly internal computations from the side of the acting agent. 
We are interested in agents subject to perceptual, computational, and storage limitations. 
Here we investigate stationary policies. 

We assume that for each stationary policy $\pi\in\Delta_{S,A}$ there is exactly one stationary world state distribution $p^\pi(w)\in\Delta_W$ and that it is attained in the limit of infinite time when running policy $\pi$, irrespective of the starting distribution $\mu$. This is a standard assumption that holds true, for instance, whenever the transition kernel $\alpha$ is strictly positive. 
In this case~\eqref{eq:pertimeexpectedreward} can be written as 
\begin{equation}
\mathcal{R}(\pi) = \sum_w p^\pi(w) \sum_{a} p^\pi(a|w) R(w,a), 
\label{eq:rew}
\end{equation}
where $p^\pi(a|w) = \sum_{s}\pi(a|s) \beta(s|w)$. 
An optimal stationary policy is a policy $\pi^\ast\in\Delta_{S,A}$ with $\Rcal(\pi^\ast)\geq \Rcal(\pi)$ for all $\pi\in\Delta_{S,A}$. 
Note that maximizing~\eqref{eq:rew} over $\Delta_{S,A}$ is the same as maximizing the discounted expected reward~\eqref{eq:discountedexpectedreward} over $\Delta_{S,A}$ with $\mu(w)=p^\pi(w)$~\citep[see][]{singh1994learning}. 
%See Appendix~\ref{appendix:objective} for a justification. 
The expected reward per time step appears more natural for POMDPs than the discounted expected reward, because, assuming ergodicity, it is independent of the starting distribution, which is not directly accessible to the agent. 
Our discussion focusses on average rewards, but our main Theorem~\ref{theorem:POMDP} also covers discounted rewards. 

Our analysis is motivated by the following natural question: 
Given that every MDP has a stationary deterministic optimal policy, 
does every POMDP have an optimal stationary policy with small stochasticity? 
Bounding the required amount of stochasticity for a class of POMDPs would allow us to define a policy model $\mathcal{M}\subseteq\Delta_{S,A}$ with 
\begin{equation}
\max_{\pi\in\Delta_{S,A}} \mathcal{R}(\pi) = \max_{\pi\in\mathcal{M}} \mathcal{R}(\pi),
\label{eq:optidom}
\end{equation}
for every POMDP from that class. 
We will show that such a model $\mathcal{M}$ can be defined in terms of the number of ambiguous sensor states and actions, 
such that $\mathcal{M}$ contains optimal stationary policies for all POMDPs with that number of actions and ambiguous sensor states.  
Depending on this number, $\mathcal{M}$ can be much smaller in dimension than the set of all stationary policies. 

The following examples illustrate some cases where optimal stationary control requires stochasticity and some of the intricacies involved in upper bounding the necessary amount of stochasticity. 

\begin{figure}
	\centering
\begin{tabular}{cc}	
\includegraphics[scale=.8]{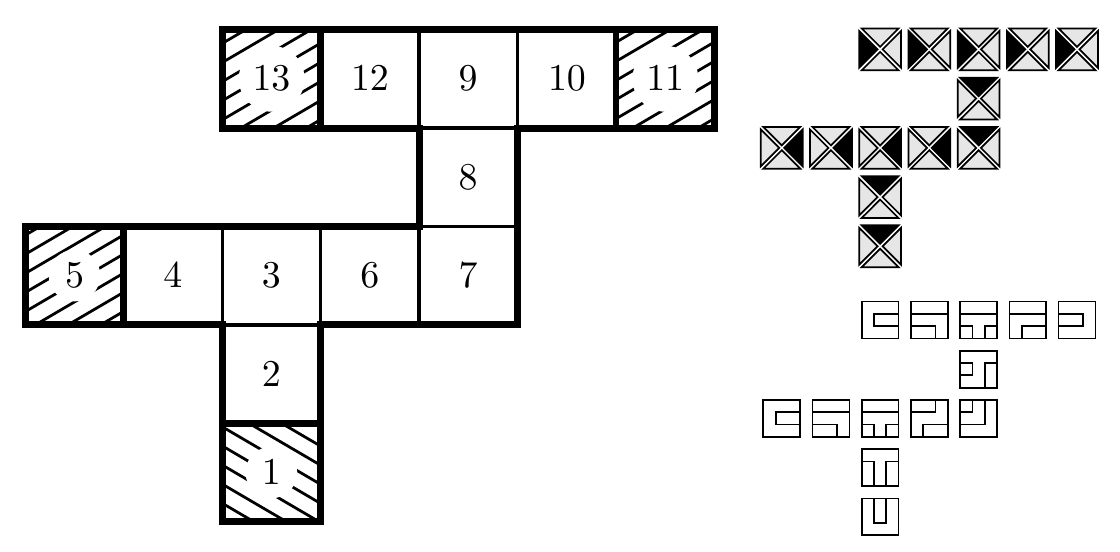} 
& 
\includegraphics[scale=.9]{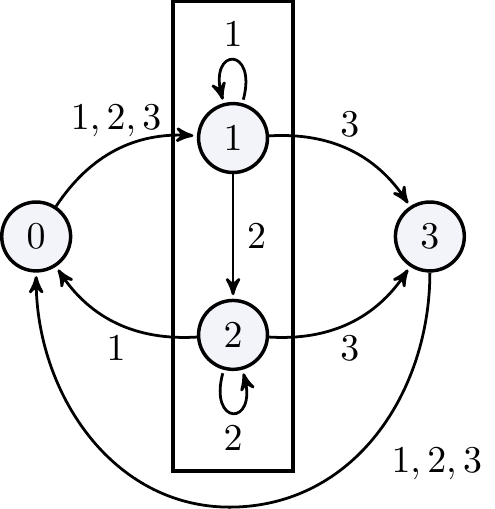}\\
(a) & (b)
\end{tabular}
\caption{
(a) Illustration of the maze Example~\ref{sec:example}. The left part shows the configuration of world states. 
The upper right shows an optimal deterministic policy in the MDP setting. 
At each state, the policy action is in the black direction. 
The lower right shows the sensor states as observed by the agent in each world state. 
(b) State transitions from Example~\ref{example1}. 
}
\label{fig:example}
\end{figure}

\begin{example}
Consider a system with $W=\{1,\ldots, n\}$, $S=\{1\}$, and $A=\{1,\ldots, n\}$. 
The reward function $R(w,a)$ is $+1$ on $a=w$ and $-1$ otherwise. 
The agent starts at some random state. 
On state $w=i$ action $a=i$ takes the agent to some random state and all other actions leave the state unchanged. 
In this case the best stationary policy chooses actions uniformly at random. 
\end{example}

\begin{example}
\label{sec:example}
Consider the grid world illustrated in Figure~\ref{fig:example}a. 
The agent has four possible actions, \emph{north}, \emph{east}, \emph{south}, and \emph{west}, which are effective when there is no wall in that direction. 
On reaching cells $5$, $11$, and $13$ the agent is teleported to cell $1$. 
On $13$ he receives a reward of one and otherwise none. 
In an MDP setting, the agent knows its absolute position in the maze. 
A deterministic policy can be easily constructed that leads to a maximal reward, 
as depicted in the upper right. 
In a POMDP setting the agent may only sense the configuration of its immediate surrounding, as depicted in the lower right.  
In this case any memoryless deterministic policy fails. 
Cells~$3$ and~$9$ look the same to the agent. 
Always choosing the same action on this sensation will cause the agent to loop around never reaching the reward cell $13$. 
Optimally, the agent should choose probabilistically between \emph{east} and \emph{west}. 
The reader might want to have a look at the experiments treating this example in Section~\ref{sec:experiments}. 
\end{example}

\begin{example}
\label{example1}
Consider the system illustrated in Figure~\ref{fig:example}b. 
Each node corresponds to a world state $W=\{0,1,2,3\}$. 
The sensor states are $S=\{0,1,3\}$, whereby $1,2$ are sensed as $1$. 
The actions are $A=\{1,2,3\}$. 
Choosing action $1$ in state $1$ and action $2$ in state $2$ has a large negative reward. 
Choosing action $2$ in state $1$ and action $1$ in state $2$ has a large positive reward. 
Choosing action $3$ in $1,2$ has a moderate negative reward and takes the agent to state $3$. 
From state $3$ each action has a large positive reward and takes the agent to $0$. 
From state $0$ any action takes the agent to $1$ or $2$ with equal probability. 
In an MDP setting the optimal policy will choose action $2$ on $1$ and action $1$ on $2$. 
In a POMDP setting the optimal policy chooses action $3$ on~$1$. 
This shows that the optimal actions in a POMDP do not necessarily correspond to the optimal actions in the underlying MDP. 
Similar examples can be constructed where on a given sensor state it may be necessary to choose from a large set of actions at random, larger than the set of actions that would be chosen on all possibly underlying world states, were they directly observed. 
\end{example}

\section{Average reward maximization as a constrained linear optimization problem}
\label{sec:optimalcontrol}

The expression $\sum_w p(w) \sum_{a} p(a|w) R(w,a)$ 
that appears in the expected reward~\eqref{eq:rew} 
is linear in the joint distribution $p(w,a) = p(w) p(a|w)\in \Delta_{W\times A}$. 
We want to exploit this linearity. 
The difficulty is that the optimization problem is with respect to the policy $\pi$, not the joint distribution, 
and the stationary world state distribution $p^\pi(w)$ depends on the policy. 
This implies that not all joint distributions $p(w,a)$ are feasible. 
The feasible set is delimited by the following two conditions. 
\begin{itemize}

\item Representability in terms of the policy: 
\begin{equation}
p(a|w) = \sum_s \pi(a|s)\beta(s|w) ,  \quad\text{for some $\pi\in \Delta_{S,A}$. } \label{condition1}
\end{equation}
The geometric interpretation is that the conditional distribution $p(a|w)$ belongs to the polytope $G\subseteq\Delta_{W, A}$ defined as the image of $\Delta_{S, A}$ by the linear map 
\begin{equation}
f_\beta\colon \pi(a|s)\mapsto \sum_s\pi(a|s)\beta(s|w).
\label{equation:fbeta}
\end{equation} 
%\footnotetext{This linear map, denoted $f_\beta\colon \Delta_{S,A}\to \Delta_{W,A}$, takes each $\pi(a|s)$ to a corresponding $p^\pi(a|w)$. More precisely, writing $\pi(a|s)\in\Delta_{S, A}$ as a row vector, the corresponding $p^\pi(a|w) \in G$, written as a row vector, is given by $\operatorname{vec}(p^\pi(a|w)) = \operatorname{vec}(\pi(a|s))(\beta^\top \otimes I_A)$, where $\beta^\top$ is the transpose of the row stochastic matrix $\beta\in \Delta_{W, S}$, $I_A$ is the $|A|\times|A|$ identity matrix, and $\otimes$ denotes the Kronecker product.} 
%
In turn, the joint distribution $p(w,a)$ belongs to the set $F\subseteq\Delta_{W \times A}$ of joint distributions with conditionals $p(a|w)$ from the set $G$. 
In general the set $F$ is not convex, although it is convex in the marginals $p(w)$ when fixing the conditionals $p(a|w)$, and vice versa. 
We discuss the details of this constraint in Section~\ref{sec:representability}. 
%
%\footnote{For each fixed $p(a|w)$ the set $\{p(w)p(a|w) \colon  p(w)\in\Delta_W\}\subseteq\Delta_{W\times A}$ is convex, and for each fixed $p(w)$ the set $\{ p(w)p(a|w) \colon p(a|w)\in G\}\in\Delta_{W \times A}$ is convex.}

\item Stationarity of the world state distribution: 
\begin{equation}
%p(w, w') = 
\sum_a p(w, a) \alpha(w'|w,a) \in \Xi, \label{condition2}
\end{equation}
where $\Xi\subseteq \Delta_{W\times W}$ is the polytope of distributions $p(w,w')$ with equal first and second marginals, $\sum_w p(w,\cdot) = \sum_{w'}p(\cdot,w')$. 
This means that $p(w)$ is a stationary distribution of the Markov transition kernel $p(w'|w)$. 		
The geometric interpretation is that $p(w,a)$ belongs to the polytope $J:= f^{-1}_\alpha(\Xi)\subseteq\Delta_{W\times A}$ 
defined as the preimage of $\Xi$ by the linear map 
\begin{equation}
f_\alpha\colon p(w,a)\mapsto\sum_a p(w,a)\alpha(w'|w,a). 
\label{eq:f-alpha}
\end{equation}
%\footnote{This linear map, denoted $f_\alpha \colon \Delta_{W\times A}\to\Delta_{W\times W}$, takes each $p(w,a)$ to a corresponding $p(w,w')$ defined by $\operatorname{vec}(p(w,w')) = \operatorname{vec}(p(w,a))(I_W \odot \alpha)$, where $I_W$ is the $|W|\times|W|$ identity matrix and $\odot$ is the column block wise Kronecker product.}
%\footnote{More precisely, the image of $f_\alpha$ is the polytope $H:= f_\alpha( \Delta_{W\times A}) \subseteq \Delta_{W\times W}$, and the preimage of $\Xi$ is the polytope $J := f_\alpha^{-1}( H\cap \Xi) \subseteq\Delta_{W\times A}$.} 
We discuss the details of this constraint in Section~\ref{sec:stationarity}. 
\end{itemize}
Summarizing, the objective function $\Rcal\colon \pi\mapsto \sum_w p^\pi(w)\sum_a p^\pi(a|w)R(w,a)$ is the restriction of the linear function $p(w,a)\mapsto \sum_{w,a}p(w,a)R(w,a)$ to a feasible domain of the form $F \cap J\subseteq \Delta_{W\times A}$, where $F$ is the set of joint distributions with conditionals from a convex polytope $G$, and $J$ is a convex polytope. 
We illustrate these notions in the next example. 

\begin{example}
\label{example:extremal}
Consider the system illustrated at the top of Figure~\ref{fig:extremal}. 
There are two world states $W=\{1,2\}$, two sensor states $S=\{1,2\}$, and two possible actions $A=\{1,2\}$. 
The sensor and transition probabilities are given by  
\begin{equation*}
\beta = 
\begin{bmatrix}
1/2 & 1/2\\ 1/2 &1/2
\end{bmatrix}, 
\quad 
\alpha(\cdot|w=1,\cdot)=\left[\begin{matrix}1&0\\0&1\end{matrix}\right],
\;
\alpha(\cdot|w=2,\cdot)=\left[\begin{matrix} 1/2&1/2\\ 1/2&1/2\end{matrix}\right].  
\end{equation*}
In the following we discuss the feasible set of joint distributions. 
The policy polytope $\Delta_{S,A}$ is a square. 
The set of realizable conditional distributions of  world states given actions is the line  
\begin{equation*}
G = f_\beta(\Delta_{S,A}) = \conv\left\{
\left[\begin{matrix}1&0\\1&0\end{matrix}\right],
\left[\begin{matrix}0&1\\0&1\end{matrix}\right] 
\right\} 
\end{equation*}
inside of the square $\Delta_{W,A}$. 
The set $F$ of joint distributions with conditionals from $G$ is a twisted surface. 
This set has one copy of $G$ for every world state distribution $p(w)$. 
See the lower left of Figure~\ref{fig:extremal}.  
%
\begin{comment}
We have
\begin{equation*}
\tilde f_\alpha(\Delta_{W,A}) = \conv \left\{ \left[\begin{matrix}1&0\\1/2&1/2\end{matrix}\right],  \left[\begin{matrix}0&1\\1/2&1/2\end{matrix}\right] \right\}, 
\end{equation*}
and 
\begin{equation*}
f_\alpha(\Delta_{W\times A})
= \conv 
\left\{
\left[\begin{matrix}1&0\\0&0\end{matrix}\right], 
\left[\begin{matrix}0&1\\0&0\end{matrix}\right], 
\left[\begin{matrix}0&0\\1/2&1/2\end{matrix}\right]
\right\},    
\end{equation*}
which is a triangle inside of the tetrahedron $\Delta_{W\times W}$. 
The polytope $\Xi$ is also a triangle (discussed in detail in Example~\ref{example:Xi}). 
The intersection of these two sets is the line 
\begin{equation*}
\Xi\cap f_\alpha(\Delta_{W\times A}) 
= \conv
\left\{
\left[\begin{matrix}1&0\\0&0\end{matrix}\right], 
\left[\begin{matrix}0&1/3\\1/3&1/3\end{matrix}\right]
\right\}. 
\end{equation*} 
\end{comment}
The set $J$ of joint distributions over world states and actions that satisfy the stationarity constraint~\eqref{condition2} is the subset of $\Delta_{W\times A}$ that $f_\alpha$ maps to the polytope $\Xi$ shown in the lower right of Figure~\ref{fig:extremal}. 
This is the triangle   
\begin{equation*}
J=f^{-1}_\alpha(\Xi)
= \conv
\left\{
\left[\begin{matrix}1&0\\0&0\end{matrix}\right], 
\left[\begin{matrix}0&1/3\\0&2/3\end{matrix}\right], 
\left[\begin{matrix}0&1/3\\2/3&0\end{matrix}\right]
\right\}.  
\end{equation*}
As we will show in Lemma~\ref{proposition1}, the extreme points of $J$ can always be written in terms of extreme points of $\Delta_{W,A}$; in the present example, in terms of 
$\left[\begin{smallmatrix}1&0\\0&1\end{smallmatrix}\right]$ (or $\left[\begin{smallmatrix}1&0\\1&0\end{smallmatrix}\right]$), $\left[\begin{smallmatrix}0&1\\0&1\end{smallmatrix}\right]$, $\left[\begin{smallmatrix}0&1\\1&0\end{smallmatrix}\right]$. 
The set $F\cap J$ is a curve. 
This is the feasible domain of the expected reward $\Rcal$, viewed as a function of joint distributions over world states and actions. 

\begin{figure}[t]
\centering
\includegraphics{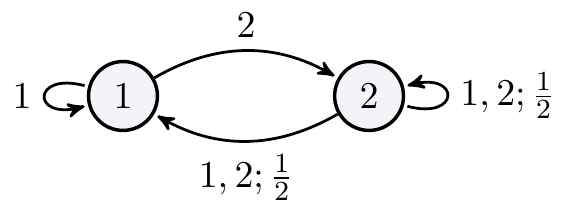}\\
\includegraphics[clip=true,trim=.5cm 1cm .5cm 1cm]{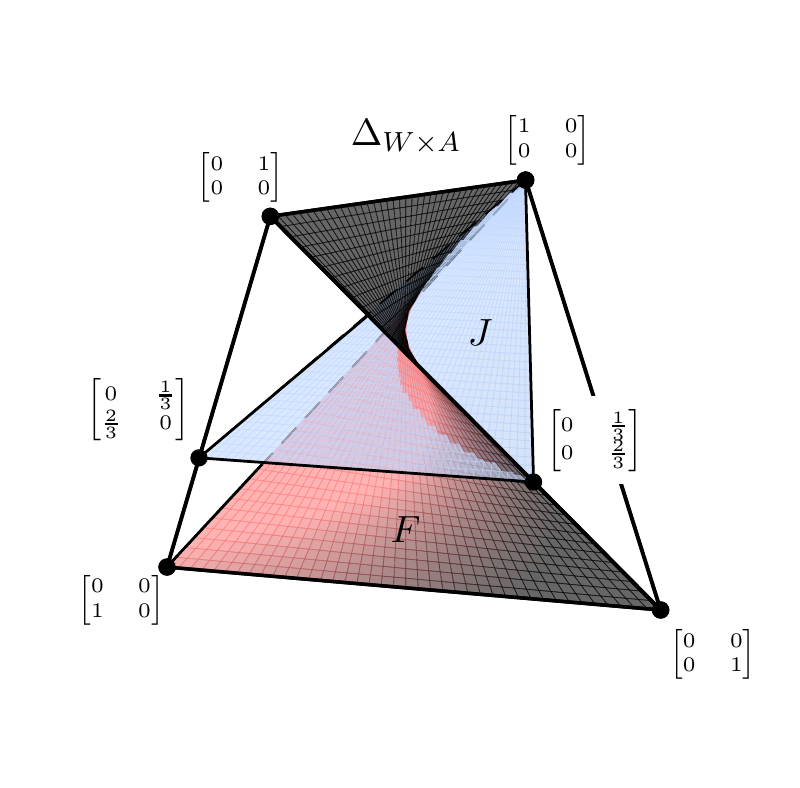}\quad
\includegraphics[clip=true,trim=.5cm 1cm .5cm 1cm]{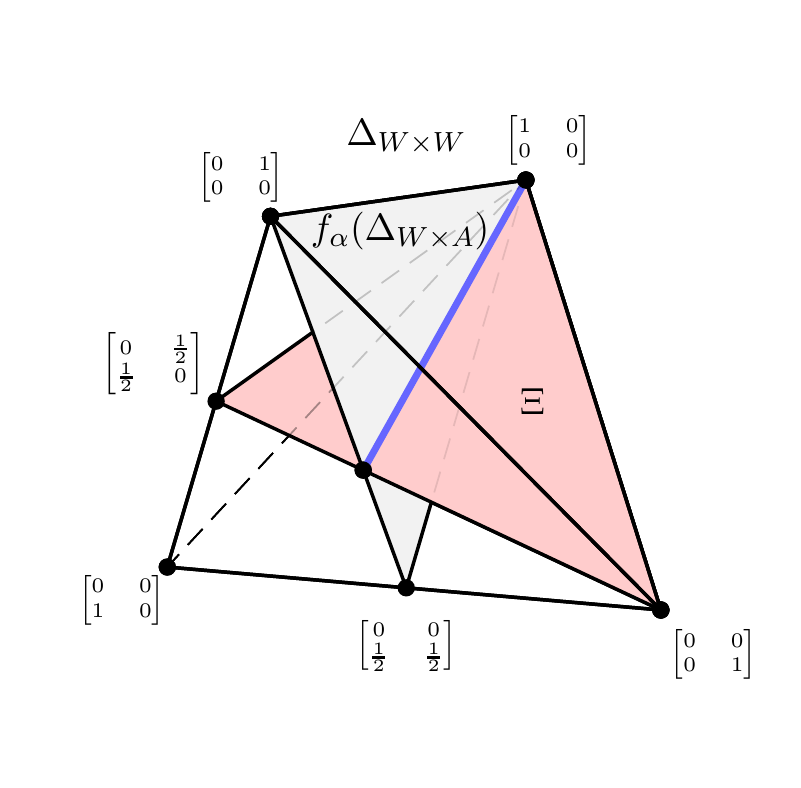}
\caption{
Illustration of Example~\ref{example:extremal}. 
The world state transitions are shown in the upper part. They are deterministic from $w=1$ and random from $w=2$. 
The lower left shows, inside of $\Delta_{W\times A}$, 
the set $F$ defined by the representability constraint~\eqref{condition1} and the polytope $J=f_\alpha^{-1}(\Xi)$ defined by the stationarity constraint~\eqref{condition2}. 
The lower right shows, inside of~$\Delta_{W\times W}$, the polytopes $f_\alpha(\Delta_{W\times A})$ and $\Xi$. %Their intersection is a line. 
}
\label{fig:extremal}
\end{figure}
\end{example}

\section{Determinism of optimal stationary policies}
\label{sec:determinism}

In this section we discuss the minimal stochasticity of optimal stationary policies. 
In order to illustrate our geometric approach we first consider MDPs and then the more general case of POMDPs. 

\begin{theorem}[MDPs]
\label{theorem:MDP}
Consider an MDP $(W,A,\alpha,R)$. 
Then there is a deterministic optimal stationary policy. 
\end{theorem}

\begin{proof}[Proof of Theorem~\ref{theorem:MDP}]
The objective function $\Rcal$ defined in Equation~\eqref{eq:rew} can be regarded as the restriction of a linear function over $\Delta_{W\times A}$ to the feasible set $J$ defined in Equation~\eqref{eq:f-alpha}. 
Since $J$ is a convex polytope, 
the objective function is maximized at one of its extreme points. 
By Lemma~\ref{proposition1}, all extreme points of $J$ can be realized by extreme points of $\Delta_{W,A}$, that is, deterministic policies. 
\end{proof}

\begin{lemma}
\label{proposition1} 
Each extreme point of $J$ can be written as $p(w,a)=p(w)p(a|w)$, where $p(w)\in \Delta_{W}$ and $p(a|w)$ is an extreme point of $\Delta_{W,A}$. 
\end{lemma}

\begin{proof}[Proof of Lemma~\ref{proposition1}]
We can view the map $f_\alpha$ from Equation~\eqref{eq:f-alpha} as taking pairs $(p(w),p(a|w))$ to pairs $(p(w),p(w'|w))$. 
Here the marginal distribution is mapped by the identity function $\Delta_W\to\Delta_W;$ $p(w)\mapsto p(w)$ and the conditional distribution by
\begin{equation*}
\tilde f_\alpha\colon \Delta_{W, A}\to \Delta_{W, W};\; p(a|w) \mapsto \sum_a p(a|w) \alpha(w'|w,a) = p(w'|w).  
\end{equation*} 

Consider some $W'\subseteq W$ for which $J$ contains a distribution $q$ whose marginal has support $W'$. 
For each $w\in W'$ let $A_{w}=\{a\in A\colon \supp(\alpha(\cdot|w,a))\subseteq W'\}$ denote the set of actions with transitions that stay in $W'$. 
With a slight abuse of notation let us write $\Delta_{W',A'} := \bigtimes_{w\in W'}\Delta_{A_{w}}\subseteq \Delta_{W',A}$ and $\Delta_{W'\times A'} := \Delta_{W'}\ast \Delta_{W',A'}=\{p(w,a)\colon p(w)\in\Delta_{W'}, p(a|w)\in\Delta_{W',A'} \} \subseteq \Delta_{W'\times A}$ for the corresponding sets of conditional and joint probability distributions. 
Note that out of $\Delta_{W\times A}$ only points from $\Delta_{W'\times A'}$ are mapped to points in $\Delta_{W'\times W'}$ and hence $J \cap \Delta_{W'\times A} \subseteq \Delta_{W'\times A'}$. 
The set $f_\alpha(\Delta_{W'\times A'})$ consists of all joint distributions $p(w,w')=p(w)p(w'|w)$ with $p(w)\in \Delta_{W'}$ and $p(w'|w)\in \tilde f_\alpha(\Delta_{W',A'})\subseteq \Delta_{W',W'}$. 
%This set has dimension 
%	\begin{equation*} 
%	\dim(f_\alpha(\Delta_{W'\times A'})) = 
%	\dim(\Delta_{W'}) + \dim(\tilde f_\alpha(\Delta_{W',A'})). 
%	\end{equation*} 
%
Now, for each conditional $p(w'|w)\in \Delta_{W',W'}$ there is at least one marginal $p(w)\in \Delta_{W'}$ such that the joint $p(w,w')\in\Delta_{W'\times W'}$ is an element of $\Xi$. 
Hence 
	\begin{equation*}
	\dim(f_\alpha(\Delta_{W'\times A'}) \cap \Xi) \geq \dim(\tilde f_\alpha(\Delta_{W',A'})). 
	\end{equation*}	
The set $J \cap \Delta_{W'\times A}$ is the union of the fibers of all points in $f_\alpha(\Delta_{W'\times A'}) \cap \Xi$. Hence 
	\begin{align*} 
	\dim(J \cap \Delta_{W'\times A})
	\geq & \dim(f_\alpha(\Delta_{W'\times A'}) \cap \Xi) + \left(\dim(\Delta_{W',A'})-\dim(\tilde f_\alpha(\Delta_{W',A'}))\right)\\
	\geq & \dim(\tilde f_\alpha(\Delta_{W',A'})) + \dim(\Delta_{W',A'})-\dim(\tilde f_\alpha(\Delta_{W',A'}))\\
	= & \dim(\Delta_{W',A'}). 
	\end{align*}

Let us now consider some extreme point $q$ of $J$. 
Suppose that the marginal of $q$ has support $W'$. 
By the previous discussion, we know that $q$ is an extreme point of the polytope $J\cap \Delta_{W'\times A'}$. 
Furthermore, $J\cap \Delta_{W'\times A'}$ is the $d$-dimensional intersection of an affine space and $\Delta_{W'\times A'}$, where $d\geq\dim(\Delta_{W',A'})$. 
%Note that $\Xi$ is the intersection of an affine space and $\Delta_{W\times W}$ and hence that the faces of $f_\alpha(\Delta_{W\times A})\cap\Xi$ and those of $J$ correspond to faces of $\Delta_{W\times A}$. 
%
This implies that $q$ lies at the intersection of $d$ facets of $\Delta_{W'\times A'}$. 
In turn $|\supp(q(w,\cdot))|=1$, for all $w\in W'$. 
This shows that $q(w,a)=p(w)p(w,a)$, where $p(w)\in\Delta_{W'}$ and $p(a|w)$ is an extreme point of $\Delta_{W',A'}$. 
We can extend this conditional arbitrarily on $w\in W\setminus W'$ to obtain a conditional that is an extreme point of $\Delta_{W,A}$. 
\end{proof}

Now we discuss the minimal stochasticity of optimal stationary policies for POMDPs. 
A policy $\pi\in\Delta_{S,A}$ is called \emph{$m$-stochastic} if it is contained in an $m$-dimensional face of $\Delta_{S,A}$. 
This means that at most $|S|+m$ entries $\pi(a|s)$ are non-zero and, in particular, that $\pi$ is a convex combination of at most $m+1$ deterministic policies. For instance, a deterministic policy is $0$-stochastic and has exactly $|S|$ non-zero entries. 
The following result holds both in the average reward and in the discounted reward settings. 

\begin{theorem}[POMDPs]
	\label{theorem:POMDP}
Consider a POMDP $(W,S,A,\alpha, \beta, R)$. 
Let $U=\{s\in S\colon |\supp(\beta(s|\cdot))|>1\}$. 
Then there is a $|U|(|A|-1)$-stochastic optimal stationary policy. 
Furthermore, for any $W, S, A$ there are $\alpha,\beta,R$ such that every optimal stationary policy is at least $|U|(|A|-1)$-stochastic.  
\end{theorem}

\begin{proof}[Proof of Theorem~\ref{theorem:POMDP}]
Here we prove the statement for the average reward setting using the geometric descriptions from Section~\ref{sec:optimalcontrol}. 
We cover the discounted setting in Section~\ref{sec:discounted} using value functions and a policy improvement argument. 

Consider the sets $G=f_\beta(\Delta_{S,A})\subseteq\Delta_{W,A}$ and $F=\Delta_W\ast G\subseteq\Delta_{W\times A}$ from Equation~\eqref{equation:fbeta}. 
We can write $G$ as a union of Cartesian products of convex sets, as $G=\bigcup_{\theta\in\Theta} G_\theta$, with $\dim(G)-\dim(G_\theta)\leq \dim(\Theta) = |U|(|A|-1)$. See Proposition~\ref{proposition2} for details.
In turn, we can write $F=\bigcup_{\theta\in\Theta} F_\theta$, where each $F_\theta = \Delta_W\ast G_\theta$ is a convex set of dimension $\dim(F_\theta) = \dim(\Delta_W) + \dim(G_\theta)$. See Proposition~\ref{proposition:convexset} for details. 

The objective function $\mathcal{R}$ is linear over each polytope $F_\theta\cap J$ and is maximized at an extreme point of one of these polytopes. 
If $F_\theta\cap J\neq \emptyset$, then each extreme point of $F_\theta\cap J$ can be written as $p(w,a)=p(w)p(a|w)$, where $p(a|w)$ is an extreme point of $G_\theta$. 
To see this, note that the arguments of Lemma~\ref{proposition1} still hold when we replace $J$ by $F_\theta\cap J$ and $\Delta_{W,A}$ by $G_\theta$. 
Each extreme point of $G_\theta$ lies at a face of $G$ of dimension at most $|U|(|A|-1)$. See Proposition~\ref{proposition2} for details. 
Now, since $f_\beta$ is a linear map, the points in the $m$-dimensional faces of $G$ have preimages by $f_\beta$ in $m$-dimensional faces of $\Delta_{S,A}$. 
Thus, there is a maximizer of $\Rcal$ that is contained in a $|U|(|A|-1)$ face of $\Delta_{S,A}$. 

The second statement, regarding the optimality of the stochasticity bound, follows from Proposition~\ref{proposition:optimal3}, which computes the optimal stationary policies of a class of POMDPs analytically. 
\end{proof}

\begin{remark} \mbox{}
\begin{itemize}
\item 
Our Theorem~\ref{theorem:POMDP} also has an interpretation for non-ergodic systems: Among all pairs $(\pi,p^\pi(w))$ of stationary policies and associated stationary world state distributions, the highest value of $\sum_wp^\pi(w)\sum_ap^\pi(a|w)R(w,a)$ is attained by a pair where the policy $\pi$ is $|U|(|A|-1)$-stochastic. 
However, this optimal stationary average reward is only equal to~\eqref{eq:pertimeexpectedreward} for start distributions $\mu$ that converge to $p^\pi(w)$. 
%This means that there may exist starting distributions for which some policy stationary distribution optimizers of $\Rcal$ perform poorly. 

\item 
For MDPs the set $U$ is empty and the statement of Theorem~\ref{theorem:POMDP} recovers Theorem~\ref{theorem:MDP}. 

\item 
In a reinforcement learning setting the agent does not know anything about the world state transitions $\alpha$ nor the observation model $\beta$ a priori, beside from the sets $S$ and $A$. 
In particular, he does not know the set $U$ (nor its cardinality). 
Nonetheless, he can build a hypothesis about $U$ on the basis of observed sensor states, actions, and rewards. 
This can be done using a suitable variant of the Baum-Welch algorithm or inexpensive heuristics, 
without estimating the full kernels $\alpha$ and $\beta$. 
\end{itemize}
\end{remark}

\section{Application to defining low dimensional policy models}
\label{sec:models}

By Theorem~\ref{theorem:POMDP}, there always exists an optimal stationary policy in a $|U|(|A|-1)$-dimensional face of the policy polytope $\Delta_{S,A}$. 
Instead of optimizing over the entire set $\Delta_{S,A}$, 
we can optimize over a lower dimensional subset that contains the $|U|(|A|-1)$-dimensional faces. 
In the following we discuss various ways of defining a differentiable policy model with this property. 

%We denote the set of deterministic policies by 
%\begin{equation}
%C_0 : = \left\{ \pi^f(a|s) = \delta_{f(s)}(a) \colon f\in  A^S \right\}. 
%\end{equation}
%Here $A^S$ denotes the set of (deterministic) functions $f\colon S\to A$. 
%This is a finite set of cardinality $|C_0| = |A|^{|S|}$. 

We denote the set of policies in $m$-dimensional faces of the polytope $\Delta_{S,A}$ by 
\begin{equation*}
C_m := \{\pi\in \Delta_{S,A}\colon  \supp(\pi)\leq |S| + m \}. 
\end{equation*}
Note that each policy in $C_m$ can be written as the convex combination of $m+1$ or fewer deterministic policies. 
For example, $C_0=\{\pi^f(a|s)=\delta_{f(s)}(a)\colon f\in A^S \}$ is the set of deterministic policies, 
and $C_{|S|(|A|-1)} = \Delta_{S,A}$ is the entire set of stationary policies. 

\paragraph{Conditional exponential families}
An exponential policy family is a set of policies of the form
\begin{equation*}
\pi_\theta(a|s) = \frac{\exp(\theta^\top F(s,a))}{\sum_{a'} \exp(\theta^\top F(s,a'))},   
\end{equation*}
where $F\colon S\times A\to\mathbb{R}^d$ is a vector of sufficient statistics and $\theta\in\mathbb{R}^d$ is a vector of parameters. 
We can choose $F$ suitably, such that the closure of the exponential family contains $C_m$. 

The \emph{$k$-interaction model} is defined by the sufficient statistics 
\begin{equation*}
F_\lambda(x) = \prod_{i\in\lambda} (-1)^{x_i}, \quad x\in\{0,1\}^n,\quad\lambda\subseteq \{1,\ldots, n\}, 1\leq |\lambda|\leq k.  
\end{equation*}
Here we can identify each pair $(s,a)\in S\times A$ with a length-$n$ binary vector $x\in\{0,1\}^n$, $n = \lceil \log_2(|S||A|)\rceil$. 
Since we do not need to model the marginal distribution over $S$, we can remove all $\lambda$ for which $F_\lambda(s,\cdot)$ is constant for all $s$. 
The $k$-interaction model is \emph{$(2^k-1)$-neighborly}~\citep{Kahle2010}, meaning that, for $2^k-1 \geq |S|+m$ it contains $C_m$ in its closure. 
This results in a policy model of dimension at most 
$\sum_{i=1}^{\lceil \log_2(|S|+m+1) \rceil}{\lceil\log_2(|S||A|)\rceil \choose i}$. 
Note that this is only an upper bound, both on $k$ and the dimension, and usually a smaller model will be sufficient. 

An alternative exponential family is defined by taking $F(s,a)$, $(s,a)\in S\times A$, equal to the vertices of a cyclic polytope. 
The cyclic polytope $C(N,d)$ is the convex hull of $\{x(t_1),\ldots, x(t_N)\}$, where $x(t)=[t,t^2,\ldots, t^d]^\top$, $t_1< t_2<\cdots<t_N$, $N>d\geq 2$. This results in a $\lfloor d / 2\rfloor$-neighborly model. 
Using this approach yields a policy model of dimension $2(|S|+m)$.

\paragraph{Mixtures of deterministic policies}
We can consider policy models of the form 
\begin{equation*}
\pi_\theta(a|s) = \sum_{f\in A^S} \pi^f(a|s) p_\theta(f), 
\end{equation*}
where $\pi^f(a|s) = \delta_{f(s)}(a)$ is the deterministic policy defined by the function $f\colon S\to A$ and 
$p_\theta(f)$ is a model of probability distributions over the set of all such functions. 
Choosing this as a $(m+1)$-neighborly exponential family yields a policy model which contains $C_m$ and, in fact, all mixtures of $m+1$ deterministic policies. This kind of model was proposed in~\cite{ICCN}. 

Identifying each $f\in A^S$ with a length-$n$ binary vector, $n \geq \lceil\log_2( |A|^{|S|}) \rceil$, 
and using a $k$-interaction model with $2^k-1 = m+1$ yields a model of dimension   
$\sum_{i=1}^{\lceil\log_2(m+2)\rceil}{\lceil \log_2( |A|^{|S| } )\rceil \choose i }$. 

Alternatively, we can use a cyclic exponential family for $p_\theta$, 
which yields a policy model of dimension $2(m+1)$. 
If we are only interested in modeling the deterministic policies, $m=0$, then this model has dimension two.

\paragraph{Conditional restricted Boltzmann machines}
A conditional restricted Boltzmann machine (CRBM) is a model of policies of the form 
\begin{equation*}
\pi_\theta(y|x) = \frac{1}{Z(x)}\sum_{z\in\{0,1\}^{n_{\text{hidden}}}} \exp(z^\top V x + z^\top W y + b^\top y + c^\top z),  
\end{equation*}
with parameter $\theta=\{W,V,b,c\}$, $W\in\mathbb{R}^{n_{\text{hidden}}\times n_{\text{out}}}$, $V\in\mathbb{R}^{n_{\text{hidden}}\times n_{\text{in}}}$, 
$b\in \mathbb{R}^{n_{\text{out}}}$, $c\in \mathbb{R}^{n_{\text{hidden}}}$. 
Here we identify each $s\in S$ with a vector $x\in\{0,1\}^{n_{\text{in}}}$, 
$n_{\text{in}} = \lceil\log_2|S|\rceil$, 
and each $a\in A$ with a vector $y\in\{0,1\}^{n_{\text{out}}}$, 
$n_{\text{out}} = \lceil\log_2 |A|\rceil$. 
There are theoretical results on CRBMs~\citep{montufar2015b} showing that they can represent every policy from $C_m$ whenever $n_{\text{hidden}}\geq |S| + m -1$. 
A sufficient number of parameters is thus  
$(|S|+m-1)(\lceil\log_2{|S|}\rceil + \lceil\log_2(|A|)\rceil) + \lceil\log_2(|A|)\rceil$. 

Each of these models has advantages and disadvantages. 
The CRBMs can be sampled very efficiently using a Gibbs sampling approach. 
The mixture models can be very low dimensional, but may have an intricate geometry. 
The $k$-interaction models are smooth manifolds.

\section{Experiments}
\label{sec:experiments}

We run computer experiments to explore the practical utility of our theoretical results. 
We consider the maze from Example~\ref{sec:example}. 
In this example, the set $U$ of sensor states $s$ with $|\supp(\beta(s|\cdot))|>1$ has cardinality two. 
By Theorem~\ref{theorem:POMDP}, there is a $|U|(|A|-1)=6$ stochastic optimal stationary policy.  
As a family of policy models we choose the $k$-interaction models from Section~\ref{sec:models}. 
The number of binary variables is $n=\lceil \log_2(|S||A|)\rceil =6$. 
This results in a sufficient statistics matrix with $64$ columns, out of which we keep only the first $40$, one for each pair $(s,a)$. 
For $k=1,\ldots, 5$, the resulting model dimension is $2, 11, 23, 29, 30$. 
The policy polytope $\Delta_{S,A}$ has dimension $|S|(|A|-1)=30$. 

\newcommand\sca{.3}
\newcommand\scb{.37}
\begin{figure}[h]
%\vspace{.2cm}
\begin{tabular}{cc} 
\includegraphics[clip=true,trim=1.5cm 14.25cm 2cm 7.1cm,scale=\sca]{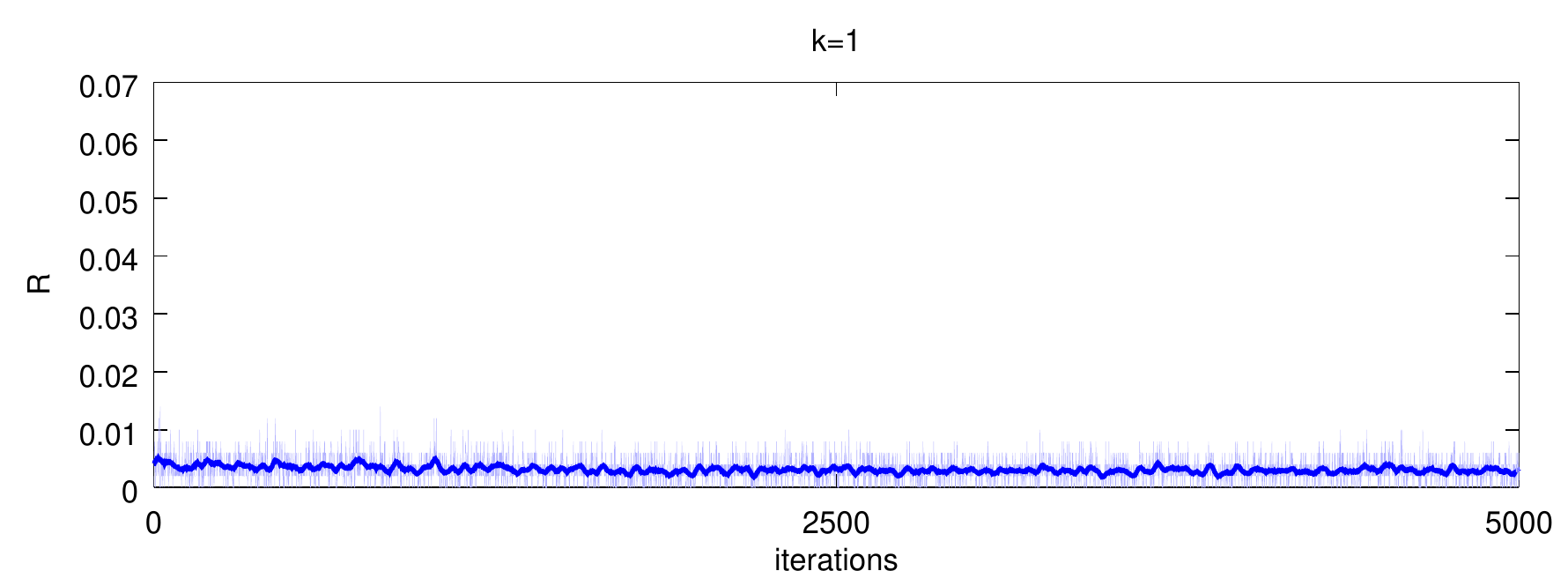}
&
\includegraphics[clip=true,trim=7.5cm 7.05cm 7.5cm 14.25cm,scale=\sca]{fig31resultsthirdrun}
\\
\includegraphics[clip=true,trim=1.5cm 14.25cm 2cm 7.1cm,scale=\sca]{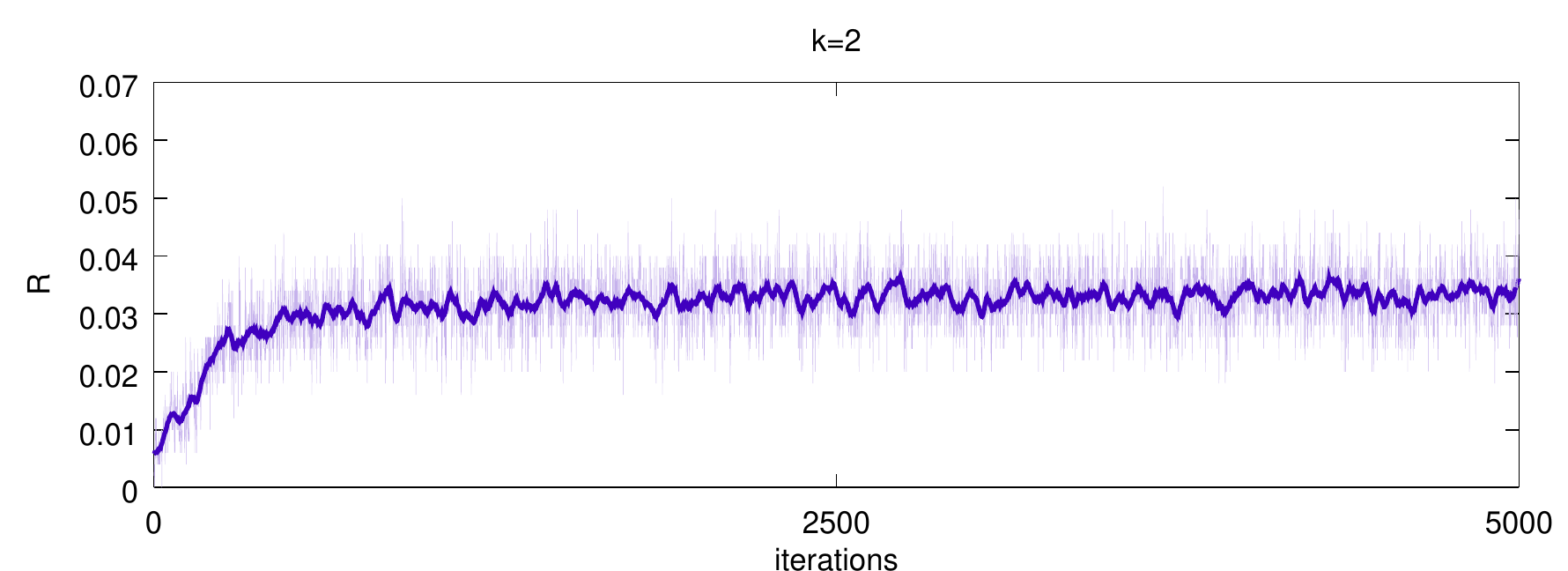}
&
\includegraphics[clip=true,trim=7.5cm 7.05cm 7.5cm 14.25cm,scale=\sca]{fig32resultsthirdrun}
\\
\includegraphics[clip=true,trim=1.5cm 14.25cm 2cm 7.1cm,scale=\sca]{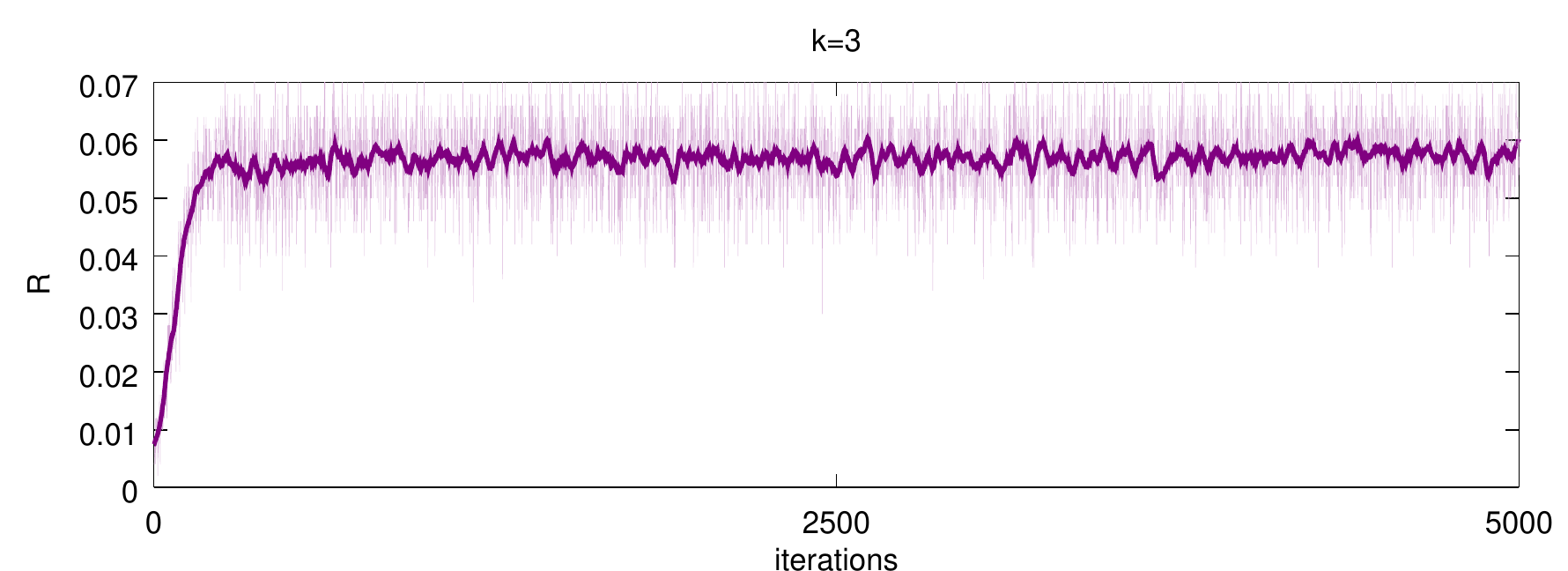}
&
\includegraphics[clip=true,trim=7.5cm 7.05cm 7.5cm 14.25cm,scale=\sca]{fig33resultsthirdrun}
\\
\includegraphics[clip=true,trim=1.5cm 14.25cm 2cm 7.1cm,scale=\sca]{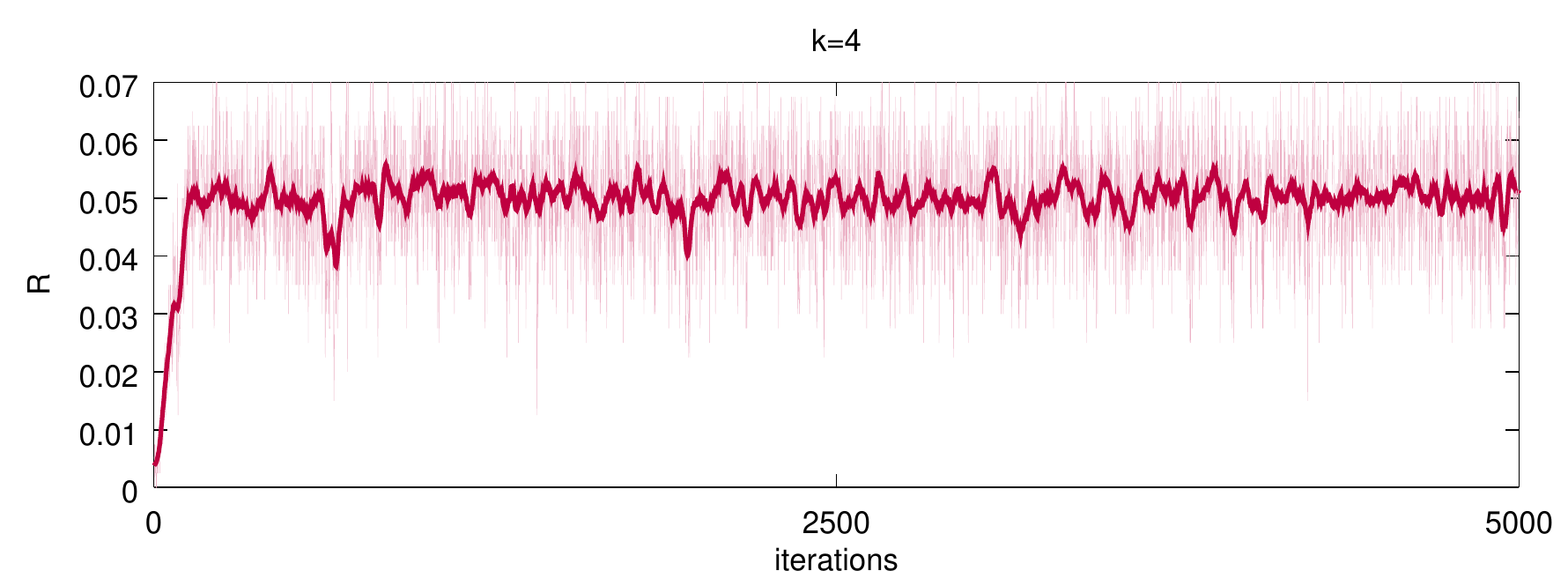}
&
\includegraphics[clip=true,trim=7.5cm 7.05cm 7.5cm 14.25cm,scale=\sca]{fig34resultsthirdrun}
\\
\includegraphics[clip=true,trim=1.5cm 14.25cm 2cm 7.1cm,scale=\sca]{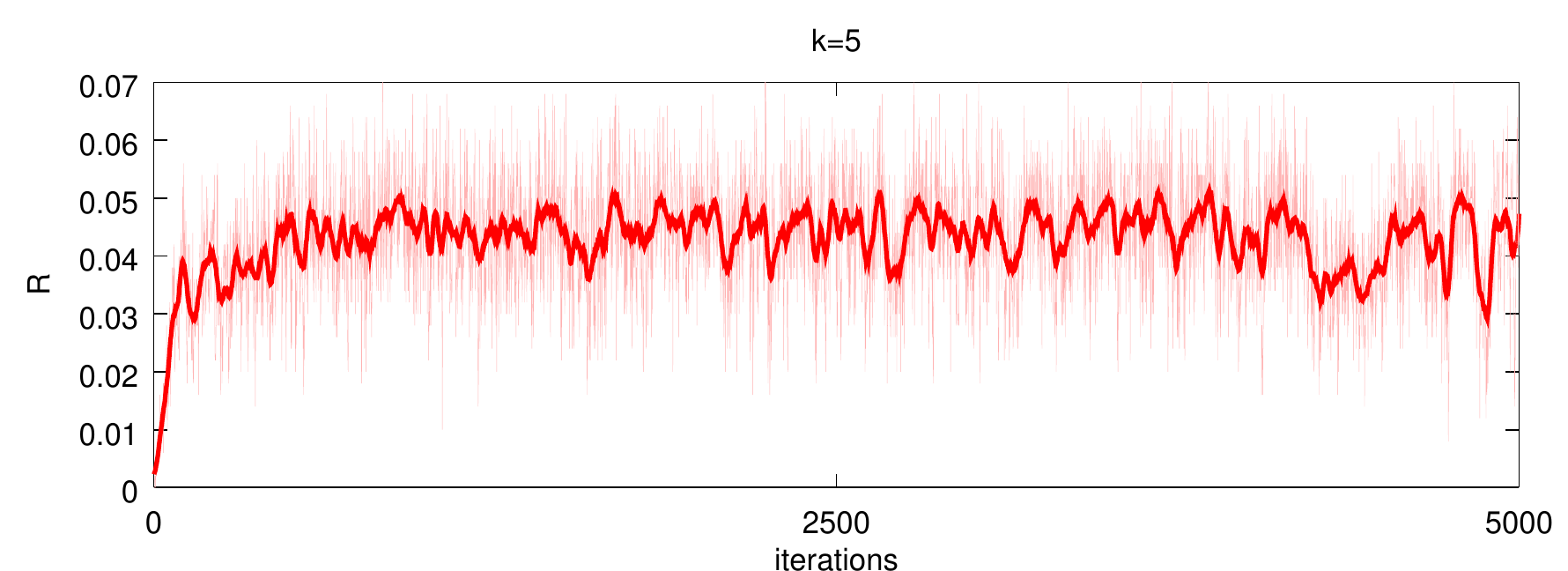}
&
\includegraphics[clip=true,trim=7.5cm 7.05cm 7.5cm 14.25cm,scale=\sca]{fig35resultsthirdrun}
\end{tabular}
\begin{tabular}{c} 
\includegraphics[clip=true,trim=1.5cm 7cm 2cm 7cm,scale=\scb]{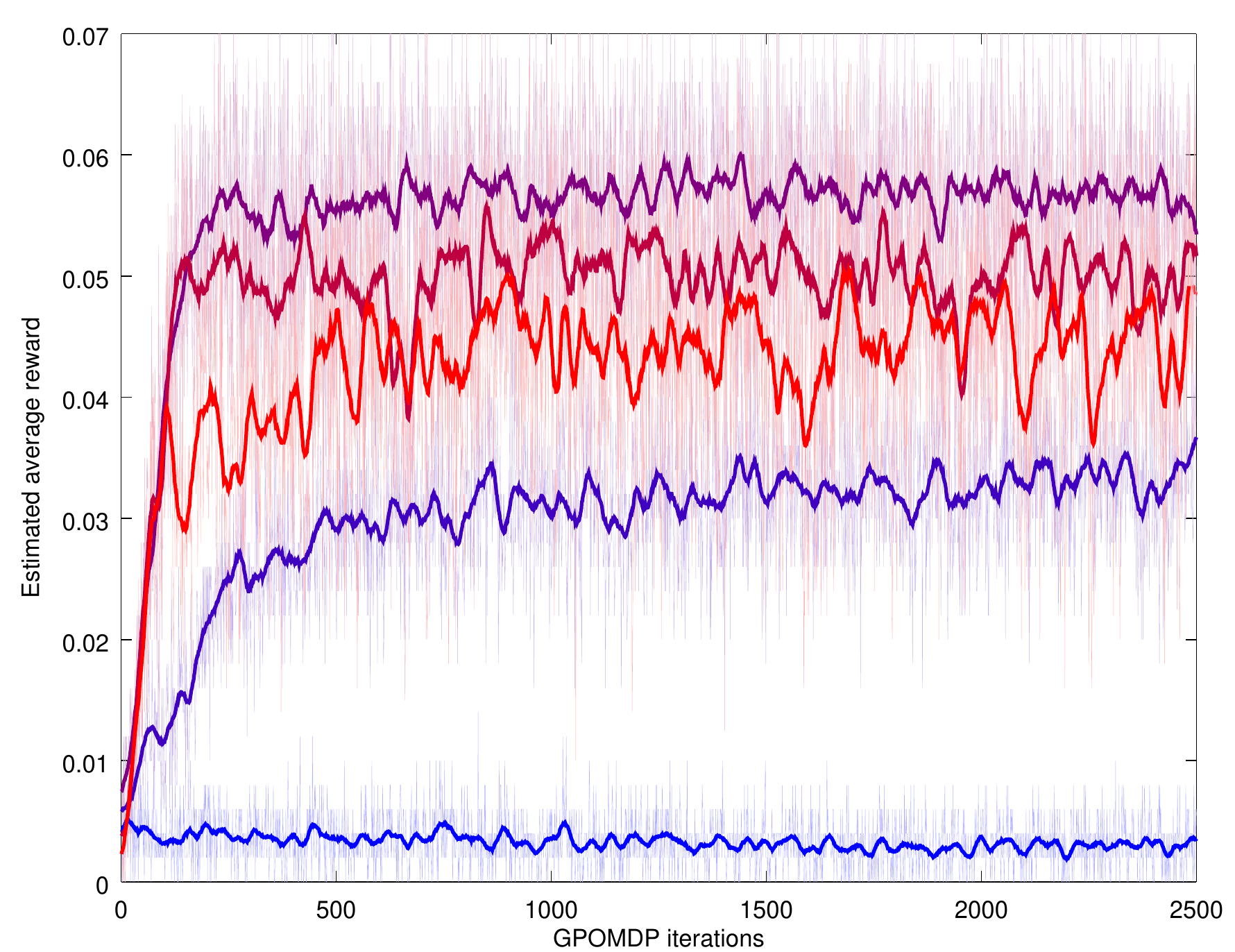}
\\
\includegraphics[clip=true,trim=1.5cm 7cm 2cm 7cm,scale=\scb]{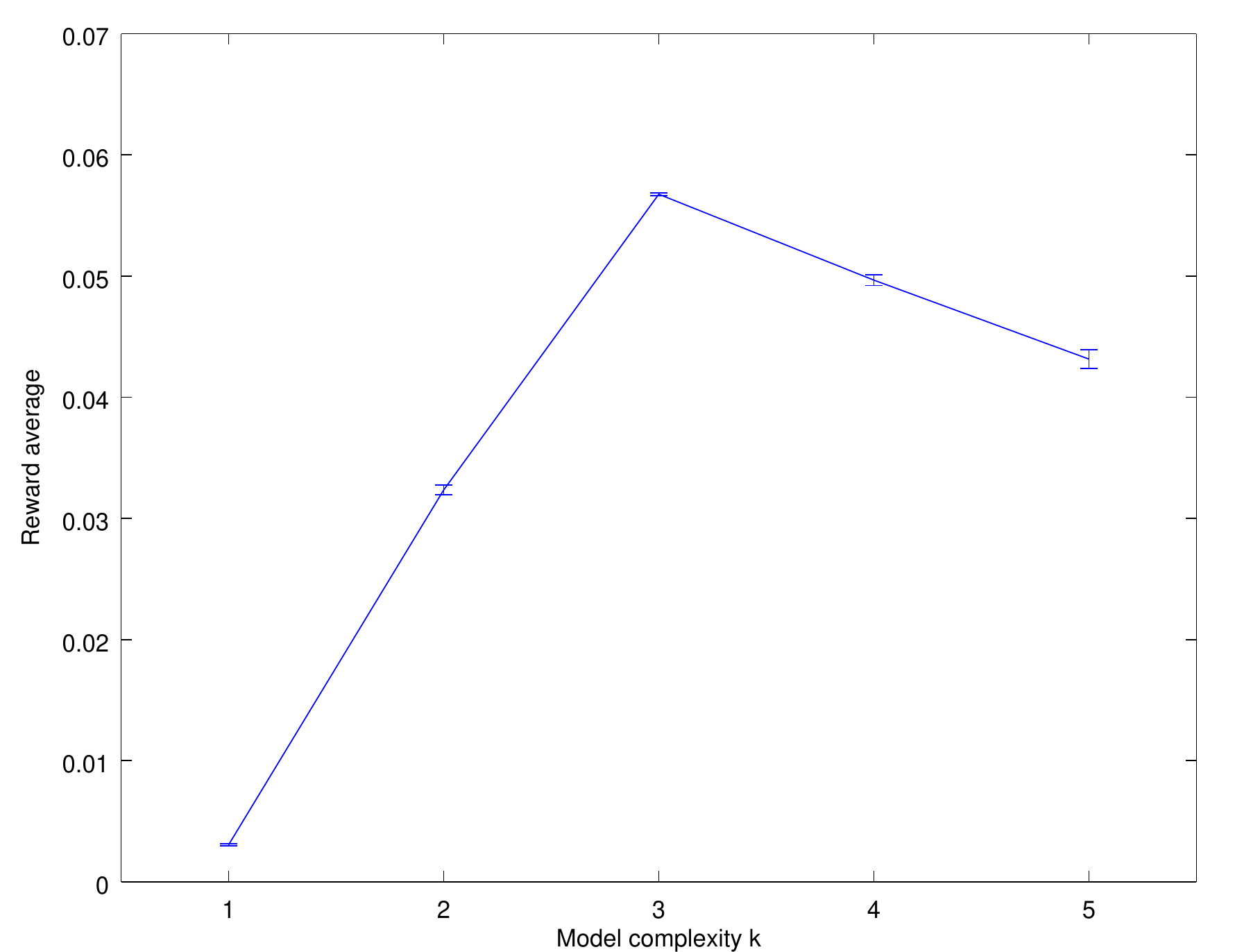}
\end{tabular}
\caption{
Experimental results on the maze Example~\ref{sec:example}. 
The left column shows the average reward learning curves for $k$-interaction models, with $k=1,\ldots, 5$ from top to bottom. 
The second column shows the final policies as matrices of sensor-state action probabilities (white is 1). 
The right column compares all learning curves and shows the overall average reward for all models. 
The model with $k=3$ performs best. 
}
\label{fig:experimentalresults}
\end{figure}

We consider the reinforcement learning problem, where the agent does not know $W, \alpha, \beta, R$ in advance. 
We use stochastic gradient with an implementation of the GPOMDP algorithm~\citep{Baxter:2001:IPE:1622845.1622855} for estimating the gradient. 
We fix a constant learning rate of $1$, 
a time window of $T=1,\ldots, 100$ for each Markov chain gradient and average reward estimation, 
and perform $10\,000$ gradient iterations on a random parameter initialization. 

The results are shown in Figure~\ref{fig:experimentalresults}. 
The first column shows the learning curves for $k=1,\ldots,5$, for the first $2\,500$ gradient iterations. 
Shown is actually the average of the learning curves for $5$ repetitions of the experiment. 
The individual curves are indeed all very similar for each fixed $k$. 
The value shown is the estimated average reward, 
with a running average shown in bold, for better visibility. 
The second column shows the final policy. 
The third column gives a detail of the learning curves and shows the reward averaged over the entire learning process. 

The independence model, with $k=1$, performs very poorly, as it learns a fixed distribution of actions for all sensor states. 
The next model, with $k=2$, performs better, but still has a very limited expressive power. 
All the other models have sufficient complexity to learn a (nearly) optimal policy, in principle. 
However, out of these, the less complex one, with $k=3$, performs best. 
This indicates that the least complex model which is able to learn an optimal policy does learn faster. 
This model has less parameters to explore and is less sensitive to the noise in the stochastic gradient.

\section{Conclusions}
\label{sec:conclusion}

Policy optimization for partially observable Markov decision processes is a challenging problem. 
Scaling is a serious difficulty in most algorithms and theoretical results are scarce on approximative methods. 
This paper develops a geometric view on the problem of finding optimal stationary policies. 
The maximization of the long term expected reward per time step can be regarded as a constrained linear optimization problem with two constraints. 
The first one is a quadratic constraint that arises from the partial observability of the world state. 
The second is a linear constraint that arises from the stationarity of the world state distribution. 
We can decompose the feasible domain into convex pieces, on which the optimization problem is linear. 
This analysis sheds light into the complexity of stationary policy optimization for POMDPs and reveals avenues for designing learning algorithms. 

We show that every POMDP has an optimal stationary policy of limited stochasticity. 
The necessary level of stochasticity is bounded above by the number of sensor states that are ambiguous about the underlying world state, independently of the specific reward function. 
This allows us to define low dimensional models which are guaranteed to contain optimal stationary policies. 
Our experiments show that the proposed dimensionality reduction does indeed allow to learn better policies faster. 
Having less parameters, these models are less expensive to train and less sensitive to noise, while at the same time being able to learn best possible stationary policies. 

\newpage 
\acks{We would like to acknowledge support from the DFG Priority Program Autonomous Learning (DFG-SPP 1527).}

\vskip 0.2in
\bibliography{referenzen}
\bibliographystyle{abbrvnat}

\appendix

\section{The representability constraint}
\label{sec:representability}

Here we investigate the set of representable policies in the underlying MDP; that is, the set of kernels of the form $p^\pi(a|w) = \sum_s \beta(s|w) \pi(a|s)$. 
This set is the image $G = f_\beta(\Delta_{S,A})$ of the linear map 
\begin{equation*}
f_\beta\colon \Delta_{S,A} \to \Delta_{W,A};\; \pi(a|s)\mapsto \sum_{s} \beta(s|w) \pi(a|s). 
\end{equation*}
We are interested in the properties of this set, depending on the observation kernel $\beta$. 

Consider first the special case of a deterministic kernel $\beta^b$, defined by $\beta^b(s|w) = \delta_{b(w)}(s)$, for some function $b\colon W\to S$. 
Then 
\begin{equation*}
f_{\beta^b}(\Delta_{S,A})  = \bigtimes_{s\in S} \sym \Delta_{b^{-1}(s), A}, 
\label{eq:sym}
\end{equation*} 
where $b^{-1}(s)\subseteq W$ is the set of world states that $b$ maps to $s$ and $\sym \Delta_{B, A} := \{ g \in \Delta_{B,A}\colon g(\cdot|w) = p, p\in\Delta_A \}$ is the set of elements of $\Delta_{B,A}$ that consist of one repeated probability distribution. 
This set can be written as a union of Cartesian products, 
\begin{equation*}
f_{\beta^b}(\Delta_{S,A}) 
= \bigcup_{\theta\in \Delta_{U,A}}\Big[ \bigtimes_{s\in U}\Big( \bigtimes_{w\in b^{-1}(s)} \theta(\cdot|s) \Big)\Big]\times
\Big[ \bigtimes_{s\in S \setminus U} \Delta_A\Big], 
\end{equation*}
where $U:= \{s \in S \colon |b^{-1}(s)|>1\}$ is the set of sensor states that can result from several world states. 
For instance, when $\beta$ is the identity function we have $G=\Delta_{W,A} = \bigtimes_{w}\Delta_A$. 

\begin{proposition}\label{proposition2}
Consider a measurement $\beta\in\Delta_{W,S}$ and the map $f_\beta\colon \Delta_{S,A}\to\Delta_{W,A};\;\pi(a|s)\mapsto \sum_s \beta(s|w)\pi(a|s)$. 
Let $U = \{s \in S \colon |\supp(\beta(s|\cdot))|>1\}$ be the sensor states that can be obtained from several world states. 
The set $G=f_\beta(\Delta_{S,A})$ can be written as 
$G = \bigcup_{\theta\in \Theta} G_\theta$, where each $G_\theta$ is a Cartesian product of convex sets, $G_\theta = \bigtimes_{w\in W} G_{\theta,w}$, $G_{\theta,w}\subseteq\Delta_A$ convex, 
and each vertex of $G_\theta$ lies in a face of $G$ of dimension at most $|U|(|A|-1)$. 
\end{proposition}

\begin{proof}[Proof of Proposition~\ref{proposition2}]
We use as index set $\Theta$ the set of policies $\Delta_{U,A}$. 
We can write 
\begin{gather*}
G = \bigcup_{\theta\in\Delta_{U,A}} G_\theta
\qquad\text{with}\qquad
G_\theta = \bigtimes_{w\in W} \left( \sum_{s\in U}\beta(s|w)\theta(\cdot|s) + \sum_{s\in S\setminus U}\beta(s|w)\Delta_A \right). 
\end{gather*}
This proves the first part of the claim. 
For the second part, note that all $G_\theta$ are equal but for addition of a linear projection of $\theta\in\Delta_{U,A}$. 
\end{proof}

\begin{figure}[t]
	\centering
	\includegraphics{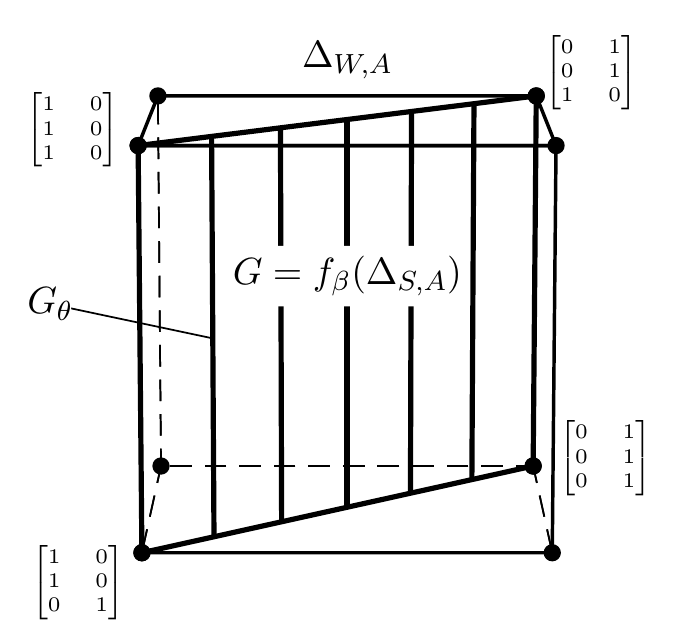}
	\caption{Illustration of Example~\ref{example:prod}. 
		Shown is a decomposition of $G=f_\beta(\Delta_{S,A})\subseteq\Delta_{W,A}$ into a collection of Cartesian products of convex sets, $G_\theta$, $\theta\in \Theta$. 
	}
	\label{figure:convexrepresentable}
\end{figure}

\begin{example}
\label{example:prod}
	Let $W=\{0,1,2\}$, $S=\{0,1\}$, $A=\{0,1\}$. 
	Let $\beta$ map $w=0$ and $w=1$ to $s=0$, and $w=2$ to $s=1$, with probability one. 
	Written as a table $(\beta(s|w))_{w,a}$ this is 
	\begin{equation*}
	\beta=
	\left[\begin{matrix}
	1 & 0 \\ 1 & 0 \\ 0 & 1
	\end{matrix}\right].
	\end{equation*} 
	The policy polytope $\Delta_{S,A}$ is the square with vertices 
	\begin{equation*}
	\left[\begin{matrix}
	1 & 0\\ 1& 0
	\end{matrix}\right], 
	\left[\begin{matrix}
	0 & 1\\ 1 & 0 
	\end{matrix}\right], 
	\left[\begin{matrix}
	0 & 1\\ 0 & 1 
	\end{matrix}\right], 
	\left[\begin{matrix}
	1 & 0\\ 0 & 1 
	\end{matrix}\right]. 
	\end{equation*}
The polytope $G = f_\beta(\Delta_{S,A})\subseteq \Delta_{W,A}$ is the square with vertices 
	\begin{equation*}
	\left[\begin{matrix}
	1 & 0\\1 & 0\\ 1& 0
	\end{matrix}\right], 
	\left[\begin{matrix}
	0 & 1\\0 & 1\\ 1 & 0 
	\end{matrix}\right], 
	\left[\begin{matrix}
	0 & 1\\0 & 1\\ 0 & 1 
	\end{matrix}\right], 
	\left[\begin{matrix}
	1 & 0\\1 & 0\\ 0 & 1 
	\end{matrix}\right]  
	\end{equation*}
and can be written as a union of Cartesian products of convex sets, illustrated in Figure~\ref{figure:convexrepresentable},     
	\begin{equation*}
	G = \bigcup_{\theta\in\Delta_A} G_\theta,\qquad
	G_\theta = \{ \theta \} \times \{ \theta \} \times \Delta_A. 
	%,\quad \theta\in\Delta_A \cong [0,1]. 
	\end{equation*}
		\begin{comment}
		The polytope $\Delta_{W,A}$ is the cube with vertices 
		\begin{equation*}
		\left[\begin{matrix}
		1 & 0 \\ 1 & 0\\ 1 & 0
		\end{matrix}\right],
		\left[\begin{matrix}
		0 & 1 \\ 1 & 0\\ 1 & 0
		\end{matrix}\right],
		\left[\begin{matrix}
		1 & 0 \\ 0 & 1\\ 1 & 0
		\end{matrix}\right],	
		\left[\begin{matrix}
		0 & 1 \\ 0 & 1\\ 1 & 0
		\end{matrix}\right],	
		\left[\begin{matrix}
		1 & 0 \\ 1 & 0\\ 0 & 1
		\end{matrix}\right],	
		\left[\begin{matrix}
		0 & 1 \\ 1 & 0\\ 0 & 1
		\end{matrix}\right],
		\left[\begin{matrix}
		1 & 0 \\ 0 & 1\\ 0 & 1
		\end{matrix}\right],	
		\left[\begin{matrix}
		0 & 1 \\ 0 & 1\\ 0 & 1
		\end{matrix}\right]. 
		\end{equation*}
		\end{comment}	
	\end{example}

As mentioned in Section~\ref{sec:optimalcontrol}, 
the set $F\subseteq\Delta{W\times A}$ of joint distributions that are compatible with the representable conditionals $G=f_\beta(\Delta_{S,A})\subseteq\Delta_{W,A}$, may not be convex. 
In the following we describe large convex subsets of $F$, depending on the properties of~$G$. 
We use the following definitions. 

\begin{definition}
	\begin{itemize}
		\item 
		Given a set of distributions $\Pcal\subseteq\Delta_W$ and a set of kernels $\Gcal\subseteq\Delta_{W,A}$, 
		let 
		\begin{equation*}
		\Pcal\ast \Gcal : =\Big\{ q(w,a) = p(w)g(a|w) \in \Delta_{W\times A} \colon p\in\Pcal, g\in \Gcal \Big\} 
		\end{equation*}
		denote the set of joint distributions over world states and actions, 
		with world state marginals in $\Pcal$ and conditional distributions in $\Gcal$. 
		
		\item 
		For any $V \subseteq W$ let 
		\begin{equation*}
		\Delta_W(V) : = \Big\{ p\in \Delta_W \colon \supp(p):=\{w\in W\colon p(w)>0 \}\subseteq V\Big\} 
		\end{equation*} 
		denote the set of world state distributions with support in $V$. 
		
		\item 
		Given a subset $V\subseteq W$ and a set of kernels $\Gcal\subseteq\Delta_{W,A}$, 
		let 
		\begin{equation*}
		\Gcal|_V := \Big\{ h\in \Delta_{V,A}\colon h(\cdot|w) = g(\cdot|w)\text{ for all }w\in V, \text{for some }g\in \Gcal  \Big\} 
		\end{equation*} 
		denote the set of restrictions of elements of $\Gcal$ to inputs from $V$. 
	\end{itemize}
\end{definition}

The following proposition states that a set of Markov kernels which is a Cartesian product of convex sets, 
with one factor for each input, 
corresponds to a convex set of joint probability distributions. 
Furthermore, if the considered input distributions assign zero probability to some of the inputs, 
then the convex factorization property is only needed for the restriction to the positive-probability inputs. 

\begin{proposition}
	\label{proposition:convexset}
	Let $V\subseteq W$. 
	Let $\Pcal\subseteq\Delta_W(V)$ be a convex set. 
	Let $\Gcal\subseteq\Delta_{W,A}$ satisfy $\Gcal|_V = \bigtimes_{w\in V} \Gcal_w\subseteq\Delta_{V,A}$, 
	where $\Gcal_w\subseteq\Delta_A$ is a convex set for all $w\in V$. 
	Then $\Pcal\ast \Gcal\subseteq \Delta_{W\times A}$ is convex. 
\end{proposition}

\begin{proof}[Proof of Proposition~\ref{proposition:convexset}]
	We need to show that, given any two distributions $q'$ and $q''$ in $\Pcal\ast \Gcal$, and any $\lambda\in[0,1]$, 
	the convex combination $q = \lambda q' + (1-\lambda)q''$ lies in $\Pcal\ast \Gcal$. 
	This is the case if and only if $q(w,a) = p(w) g(a|w)$ for some $p\in \Pcal$ and some $g\in \Delta_{W,A}$ with $g|_V\in\Gcal|_V$. 
	We have 
	\begin{align*}
	q(w,a)&=\lambda q'(w,a) + (1-\lambda) q''(w,a) \\
	&= 
	\lambda p'(w)g'(a|w) + (1-\lambda) p''(w)g''(a|w)\\
	&= 
	(\lambda p'(w) + (1-\lambda) p''(w)) \\
	& \phantom{=\;}\times \Big(\frac{\lambda p'(w)}{\lambda p'(w) + (1-\lambda) p''(w)} g'(a|w)  
	+ \frac{(1-\lambda) p''(w)}{\lambda p'(w) + (1-\lambda) p''(w)} g''(a|w) \Big). 
	\end{align*}
	This shows that $q(w,a) = p(w) g(a|w)$, where $p(w) = \lambda p'(w) + (1-\lambda)p''(w) \in \Pcal$ and 
	$g(\cdot|w) = \lambda_w g'(\cdot|w) + (1-\lambda_w) g''(\cdot|w) \in \Gcal_w$, $\lambda_w = \frac{\lambda p'(w)}{\lambda p'(w) + (1-\lambda) p''(w)}$, for all $w\in V$. Hence $g(a|w)|_V\in \Gcal|_V$ and $q\in \Pcal\ast \Gcal$. 
\end{proof}

The set of Markov kernels $\Delta_{W,A}$ is a Cartesian product of convex sets $\Delta_{W,A} = \bigtimes_{w\in W} \Delta_A$. 
The set of joint distributions $\Delta_{W\times A} = \Delta_W \ast \Delta_{W,A}$ is a simplex, which is a convex set. 

A general set $\Gcal\subseteq\Delta_{W,A}$ is not necessarily convex, let alone a Cartesian product of convex sets. 
However, it can always be written as a union of Cartesian products of convex sets of the form 
\begin{equation*}
\Gcal = \bigcup_{\theta\in\Theta} \Gcal_\theta, \quad 
\Gcal_\theta = \bigtimes_{w\in W} \Gcal_{\theta,w}, \quad
\Gcal_{\theta,w} \subseteq\Delta_A \text{ convex}. 
\end{equation*}
For instance, one can always use 
$\Theta = \Gcal$, $\Gcal_{\theta =g } = \{g\}$, $\Gcal_{\theta=g , w} = \{g(\cdot|w)\}$. 
Proposition~\ref{proposition:convexset}, together with this observation, implies that 
given any $\Gcal\subseteq\Delta_{W,A}$ and a convex set $\Pcal\subseteq\Delta_W$, 
the set of joint distributions $\Fcal = \Pcal \ast \Gcal\subseteq\Delta_{W\times A}$ is a union of convex sets $\Fcal_\theta = \Pcal \ast \Gcal_\theta$, $\theta\in \Theta$. 
The situation is illustrated in Example~\ref{example:conv}. 

\begin{example}
\label{example:conv}
Consider the settings from Example~\ref{example:prod}. 
	The set $F = \Delta_W \ast G\subseteq \Delta_{W\times A}$ is the union of following sets:
	\begin{equation*}
	F_\theta = \Delta_{W}\ast G_\theta, \quad \theta \in \Delta_A. 
	\end{equation*}
	Each $F_\theta\subseteq F\subseteq \Delta_{W\times A}$ is a polytope with vertices
	\begin{equation*}
	\left[\begin{matrix}
	\theta & (1-\theta)\\0 & 0 \\ 0 & 0 
	\end{matrix}\right], 
	\left[\begin{matrix}
	0 & 0 \\ \theta & (1-\theta)\\ 0 & 0 
	\end{matrix}\right], 
	\left[\begin{matrix}
	0 & 0 \\ 0 & 0 \\ 1 & 0 
	\end{matrix}\right], 
	\left[\begin{matrix}
	0 & 0 \\ 0 & 0 \\ 0 & 1
	\end{matrix}\right]. 
	\end{equation*}
\end{example}

\section{The stationarity constraint}
\label{sec:stationarity}

In the objective function, the marginal distribution over world states is the stationary distribution of the world state transition kernel, and not some arbitrary distribution over world states. 
The coupling of transition kernels and marginal distributions can be described in terms of the polytope $\Xi$ of joint distributions in $\Delta_{W\times W}$ with equal first and second marginals. This is given by 
\begin{equation*}
\Xi := \Big\{ p(w,w')\in \Delta_{W\times W}\colon \sum_{w'} p(\cdot, w') = \sum_w p(w,\cdot) \Big\}. 
\end{equation*} 
The second marginal is the result of applying the conditional as a Markov kernel to the first marginal; that is, $\sum_{w}p(w)p(w'|w) = p(w')$. 
Hence equality of both marginals means that the marginal is a stationary distribution of the transition $p(w'|w)$. 

The polytope $\Xi$ has been studied by~\citet{Weis2010} under the name Kirchhoff polytope. 
The vertices of $\Xi$ are the joint distributions of the following form. 
For any non-empty subset $\Wcal\subseteq W$ and a cyclic permutation $\sigma \colon \Wcal\to\Wcal$, there is a vertex defined by 
\begin{equation*}
c_{\Wcal, \sigma}(w,w') := 
\frac{1}{|\Wcal|}
\begin{cases}
1, & \text{if } \sigma(w) = w'\\
0,& \text{otherwise}
\end{cases}.
\label{equation:verticesXi}
\end{equation*}
The dimension is $\dim(\Xi)= |W|(|W|-1)$. To see this, note that each strictly positive transition $p(w|w)$ is trivially a primitive Markov kernel and hence it has a unique stationary limit distribution. 
In turn, the set of strictly positive transitions, which has dimension  has dimension $|W|(|W|-1)$, corresponds to the relative interior of $\Xi$.  

\begin{example} 
	\label{example:Xi}
	Let $W=\{0,1\}$. 
	The non-empty subsets of $W$ are $\mathcal{W}=\{0\}, \{1\},\{0,1\}$, 
	and the cyclic permutations of these subsets are $0\to 0$, $1\to 1$, $(0, 1) \to (1,0)$. 
	The Kirchhoff polytope $\Xi$ is the triangle enclosed by the points   
	\begin{equation*}
	c_{\{0\},0\to 0 } = \left[\begin{matrix}
	1 & 0 \\ 0 & 0
	\end{matrix}\right],\quad
	c_{\{1\},1\to 1 } =
	\left[\begin{matrix}
	0 & 0 \\ 0 & 1
	\end{matrix}\right],\quad
	c_{\{0,1\},(0,1)\to(1,0) } =
	\left[\begin{matrix}
	0 & 1/2 \\ 1/2 & 0
	\end{matrix}\right].  
	\end{equation*}
	Every strictly positive joint distribution $p(w,w')$ corresponds to a marginal $p(w)$ and a conditional distribution $p(w'|w)$. 
	Each point in the interior of $\Xi$ corresponds to a point in the interior of $\Delta_{W,W}$. 
	The situation is illustrated in Figure~\ref{figure:marginals}. 
\end{example}

\begin{figure}
	\centering
	\includegraphics{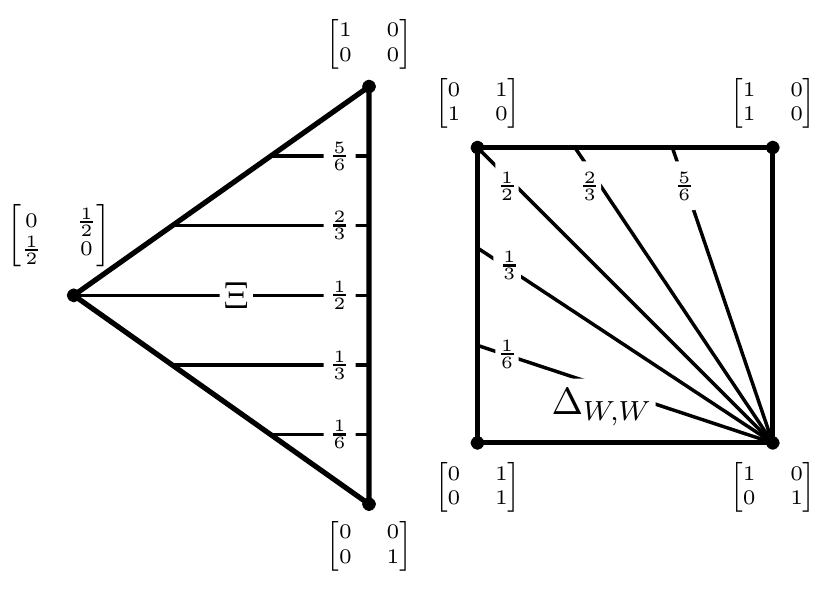}  
	\caption{The polytope $\Xi\subseteq\Delta_{W\times W}$, $W=\{0,1\}$, discussed in Example~\ref{example:Xi}. 
		Subsets with first marginals satisfying $p(w=0)=\tfrac16, \ldots, \tfrac56$ are highlighted. 
		The right panel shows the corresponding sets of conditional distributions in $\Delta_{W,W}$. 
	}
	\label{figure:marginals}
\end{figure}

\section{Determinism of optimal stationary policies for discounted rewards}
\label{sec:discounted}

\begin{theorem}\label{theorem:POMDP2}
Consider a POMDP $(W,S,A,\alpha,\beta,R)$, a discount factor $\gamma\in(0,1)$, and a start distribution $\mu$. 
Then there is a stationary policy $\pi^\ast\in \Delta_{S,A}$ that is deterministic on each $s\in S$ with $|\supp(\beta(s|\cdot))|\leq 1$ and satisfies $\Rcal_{\mu}^\gamma(\pi^\ast)\geq \Rcal_{\mu}^\gamma(\pi)$ for all $\pi\in \Delta_{S,A}$. 
\end{theorem}

We will prove Theorem~\ref{theorem:POMDP2} using a policy improvement argument. 
The world state value function of a policy $\pi$ is the unique solution of the Bellman equation 
\begin{equation*}
V^\pi(w) = \sum_a p^\pi(a|w)\Big[R(w,a) +  \gamma \sum_{w'}\alpha(w'|w,a) V^\pi(w')\Big]. 
\end{equation*}
The action value function is given by 
\begin{equation*}
Q^\pi(w,a) = R(w,a) + \gamma \sum_{w'}\alpha(w'|w,a)V^\pi(w'). 
\end{equation*}
These definitions make sense both for MDPs and POMDPs. However, while for MDPs there is a stationary policy that maximizes the value of each world state simultaneously, for POMDPs the same is not true in general.

\begin{lemma}[POMDP policy improvement]
\label{lemma:policyimprovement}
Let $\pi, \pi' \in \Delta_{S,A}$ be two policies with 
$\sum_a p^{\pi'}(a|w) Q^\pi(w,a) \geq V^\pi(w)$ for all $w$. 
Then $V^{\pi'}(w)\geq V^{\pi}(w)$ for all $w$. 
\end{lemma}

\begin{proof}[Proof of Lemma~\ref{lemma:policyimprovement}]
The proof follows closely the arguments of the MDP deterministic policy improvement theorem presented by~\cite{suttonbarto98}. 
\begin{align*}
V^\pi(w) 
\leq& \sum_a p^{\pi'}(a|w)Q^\pi(w,a)\\
=& \mathbb{E}_{\pi',w_0=w}\Big[R(w_0, a_0) + \gamma V^\pi(w_1) \Big]\\
\leq& \mathbb{E}_{\pi',w_0=w}\Big[R(w_0, a_0) + \gamma \sum_{a}p^{\pi'}(a|w_1)Q^\pi(w_1,a) \Big]\\
=& \mathbb{E}_{\pi',w_0=w}\Big[R(w_0, a_0) + \gamma R(w_1,a_1) + \gamma^2 V^\pi(w_2) \Big]\\
\leq&\lim_{T\to\infty} \mathbb{E}_{\pi',w_0=w}\Big[\sum_{t=0}^{T-1} \gamma^t R(w_t, a_t) \Big] 
=  V^{\pi'}(w).\qedhere
\end{align*}
\end{proof}

\begin{proof}[Proof of Theorem~\ref{theorem:POMDP2}]
Consider any policy $\pi\in\Delta_{S,A}$. 
%The corresponding value function $V^\pi(w)$ is the unique solution of the linear system~\eqref{eq:valfun1}. 
Consider some $\tilde s\in S$ with $\supp(\beta(\tilde s|\cdot))=\tilde w$ and  
$\tilde a \in \argmax_a Q^\pi(\tilde w,a)$. 
We define an alternative policy by $\pi'(a|s)=\pi(a|s)$, $s\neq \tilde s$, and $\pi'(\tilde a| \tilde s)=1$. 
This policy is deterministic on $\tilde s$. 
We have 
\begin{equation*}
p^{\pi'}(a|w) = \sum_s\beta(s|w)\pi'(a|s) = \sum_{s\neq\tilde s}\beta(s|w)\pi'(a|s) = p^\pi(a|w), \quad \text{for all }w\neq \tilde w, 
\end{equation*} 
and 
\begin{equation*}
p^{\pi'}(a|\tilde w) 
= \sum_s\beta(s|\tilde w)\pi'(a|s) 
= \sum_{s\neq \tilde s}\beta(s|\tilde w)\pi(a|s) + \beta(\tilde s|\tilde w)\delta_{a,\tilde a}.  
\end{equation*}
In turn, 
\begin{align*}
\sum_a p^{\pi'}(a|w) Q^{\pi}(w,a) 
= \sum_a p^{\pi}(a|w) Q^{\pi}(w,a) 
= V^\pi(w),\quad\text{for all }w\neq\tilde w, 
\end{align*}
and 
\begin{align*}
\sum_a p^{\pi'}(a|\tilde w) Q^{\pi}(\tilde w,a) 
=& \sum_{a} 
\Big[\sum_{s\neq \tilde s}\beta(s|\tilde w)\pi(a|s) + \beta(\tilde s|\tilde w)\delta_{a,\tilde a} \Big]Q^\pi(\tilde w,a)\\
=&\sum_{a} 
\Big[\sum_{s\neq \tilde s}\beta(s|\tilde w)\pi(a|s)\Big]Q^\pi(\tilde w,a) 
+ \sum_a\beta(\tilde s|\tilde w)\delta_{a,\tilde a} Q^\pi(\tilde w,a)\\
\geq&\sum_{a} 
\Big[\sum_{s\neq \tilde s}\beta(s|\tilde w)\pi(a|s)\Big]Q^\pi(\tilde w,a) 
+ \sum_a\beta(\tilde s|\tilde w)\pi(a|s) Q^\pi(\tilde w,a)\\
=&\sum_{a} 
\Big[\sum_{s}\beta(s|\tilde w)\pi(a|s)\Big]Q^\pi(\tilde w,a)\\
=&\sum_a p^{\pi}(a|\tilde w) Q^{\pi}(\tilde w,a)  
= V^\pi(\tilde w). 
\end{align*}
This shows that $\sum_a p^{\pi'}(a|w)Q^\pi(w,a)\geq V^{\pi}(w)$, for all $w$. By Lemma~\ref{lemma:policyimprovement} $V^{\pi'}(w)\geq V^\pi(w)$, for all $w$. 
Repeating the same arguments, we conclude that any policy $\pi$ can be replaced by a policy $\pi'$ which is deterministic on each $s\in S$ with $|\supp\beta(s|
\cdot)|=1$ and which satisfies $V^{\pi'}(w)\geq V^\pi(w)$ for all $w\in W$. Sensor states with $|\supp(\beta(s|\cdot))|=0$ are never observed and the corresponding policy assignment immaterial. 
This completes the proof. 
\end{proof}

We conclude this section with a few remarks. 
It is worthwhile to mention the relation 
\begin{equation*}
	\sum_w p^\pi(w) V^\pi(w) = \frac{\Rcal(\pi)}{1-\gamma}, 
	\label{eq:averdisc}
\end{equation*}
which follows from~\citep[see][Fact~7]{singh1994learning} 
\begin{align*}
	\sum_w p^\pi(w) V^\pi(w)
	&= \sum_w p^\pi(w)  \sum_{a} p^\pi(a|w) \left[ R(w,a) + \gamma\sum_{w'} \alpha(w'|w,a) V^\pi(w')\right]\\
	&= \sum_w p^\pi(w) \left[ \sum_{a} p^\pi(a|w) R(w,a) +\gamma \sum_{w'}p^\pi(w'|w) V^\pi(w') \right]\\
	&= \Rcal(\pi) + \sum_w p^\pi(w) \gamma \sum_{w'}p^\pi(w'|w) V^\pi(w') \\
	&= \Rcal(\pi) +  \gamma \sum_{w'} p^\pi(w') V^\pi(w').  
\end{align*}

Note that $\Rcal_{\mu}^\gamma(\pi)=\sum_w\mu(w)V^\pi(w)$. 
Hence if two policies $\pi,\pi'$ satisfy $V^{\pi'}(w)\geq V^\pi(w)$, for all $w$, then $\Rcal_{\mu}^\gamma(\pi')\geq \Rcal_{\mu}^\gamma(\pi)$, for all $\mu$. 
However, the same hypothesis does not necessarily imply any particular relation between $\Rcal(\pi') = (1-\gamma) \sum_w p^{\pi'}(w)V^{\pi'}(w)$ and $\Rcal(\pi) = (1-\gamma) \sum_w p^{\pi}(w)V^{\pi}(w)$.

\section{Examples with analytic solutions}
\label{sec:optiex}

We discuss three examples where it is possible to compute the optimal memoryless policy analytically and show that it has stochasticity equal to the upper bound indicated in Theorem~\ref{theorem:POMDP}. 
This proves the optimality of the stochasticity bound. 
The first two examples consider the case $|U|=|S|=1$ and the third example the case with arbitrarily large $|U|$. 

\begin{example} 
	\label{exopt1}
	Consider a POMDP where the agent has only one sensor state, $K$ possible actions, and the world state transitions are as shown in Figure~\ref{fig:chain}. 
	At each world state only one action takes the agent further to the right, while all other actions take it to $w=0$ with probability one. 
	At the world state $w=K$, the agent receives a reward and all actions take it to $w=0$. 

	\begin{figure}[h]
			\centering
			\vspace{-1cm}
			\begin{tikzpicture}[->,>=stealth',shorten >=1pt,auto,node distance=3cm, thick, main node/.style={circle,fill=blue!20,draw},every loop/.style={<-,shorten <=1pt}]		
			\tikzstyle{state} = [circle, line width=1pt, draw=black, inner sep=.1cm, minimum size = .7cm, %shading=radial,outer color=bl!10,inner color=bl!5];
			fill=bl!5];
			
			\node[state] (0) at (0,0) {$0$}; 	
			\node[state] (1) at (2,0) {$1$}; 	
			\node[state] (2) at (4,0) {$2$}; 	
			\node[state] (3) at (8,0) {$K$}; 	
			\node[rectangle, draw=white, inner sep=0,fill=white] (R) at (6,0) {$\;\cdots\;$};
			
			\path[every node]
			(0) edge node [below] {$1$} (1)
			edge [loop left] node {$\hat 1$} (0);
					
			\path[every node]
			(1) edge node [below] {$2$} (2)
			edge [bend right, in=-120] node [above] {$\hat 2$} (0);
	
			\path[every node]
			(2) edge node [below] {$3$} (R)
			edge [bend right, in=-100] node [above]{$\hat 3$} (0);
					
			\path[every node]
			(R) edge node [below] {$K$} (3);
			
			\path[every node]
			(3) edge [bend right, in=-80] (0);
			\end{tikzpicture} 
			\caption{
			State transitions from Example~\ref{exopt1}. 
			The number in each node indicates the world state. The sensor state is always $s=1$. 
			At each world state there is only one action that takes the agent further to the right, while all other actions take it to $w=0$ with probability one.}
			\label{fig:chain}
		\end{figure}
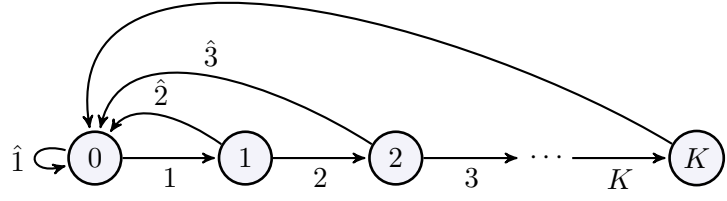
	
We will use the abbreviations $\pi_a = \pi(a|s=1)$ and $p_w = p^\pi(w)$. 
The world-state transition matrix is given by 
	\begin{equation*}
	(p(w'|w))_{w,w'} = 
	\begin{bmatrix}
	1-\pi_1 & \pi_1 & &&&\\
	1-\pi_2 && \pi_2  &&&\\
	1-\pi_3 &&& \pi_3 && \\
	\vdots    &&&&\ddots&\\
	1-\pi_K &&&&& \pi_K  \\
	1             &&&&&
	\end{bmatrix}. 
	\end{equation*} 
For this system the expected reward per time step is $\Rcal(\pi) = p^\pi(w=K)$. 
The stationary distribution of this transition matrix satisfies 
	\begin{eqnarray*}
		p_1 &=& p_0 \pi_1\\
		p_2 &=& p_1 \pi_2\\
		&\vdots&\\
		p_K &=& p_{K-1} \pi_K.
	\end{eqnarray*}
	Using the relation 
	\begin{equation*}
	1 = p_0 + p_1 +\cdots +p_K = p_0(1 + \pi_1 + \pi_1\pi_2 +\cdots +\pi_1\cdots\pi_K), 
	\end{equation*}
	we obtain 
	\begin{equation}
	p_K 
	= \frac{\pi_1\cdots\pi_K}{1 + \pi_1 + \pi_1\pi_2 +\cdots +\pi_1\cdots\pi_K}. 
	\label{equation:pK1}
	\end{equation}
	This is positive if and only if $\pi_1,\ldots, \pi_K$ are all larger than zero. 
	In turn, any optimal memoryless stochastic policy has at least $K$ positive probability actions at the single observation $s=1$. 
	The next proposition describes the precise form of the optimal memoryless policy (in this case unique). 
	\begin{proposition}
		\label{proposition:optimal1}
		The optimal memoryless policy of the POMDP Example~\ref{exopt1} is given by 
		\begin{eqnarray*}
			\pi_1 &=& c\\
			\pi_i &=& \pi_{i-1} + c \pi_1\cdots\pi_{i-1},\quad i=2,\ldots, K, 
		\end{eqnarray*} 
		where $c$ is the unique real positive solution of
		\begin{equation} 
		\pi_1+\cdots +\pi_K =1.\label{eq:ceq}
		\end{equation} 
	\end{proposition}
	
\begin{proof}[Proof of Proposition~\ref{proposition:optimal1}]
		The policy that maximizes $p_K$ can be found using the method of Lagrange multipliers. 
		The critical points satisfy 
		$1 - \sum_i\pi_i = 0$  and $\frac{\partial p_K}{\partial \pi_i} - \lambda = 0$ for all $i=1,\ldots, K$. 
		Computing the derivatives of~\eqref{equation:pK1} we find that 
		\begin{equation*}
		\frac{\pi_1\cdots\pi_K}{(1+\pi_1 + \cdots+\pi_1\cdots\pi_K)}
		\times\left(1 - \frac{\pi_1\cdots\pi_i + \cdots +\pi_1\cdots\pi_K}{1+\pi_1 + \cdots +\pi_1\cdots\pi_K}\right)\frac{1}{\pi_i} =\lambda,
		\quad\text{for $i=1,\ldots, K$}.  
		\end{equation*}
		This implies that 
		\begin{equation*}
		\pi_i = c( 1+\pi_1+\cdots +\pi_1\cdots\pi_{i-1}),\quad \text{for $i=1,\ldots, K$}, 
		\end{equation*}
		where $c = \lambda^{-1}\frac{\pi_1\cdots\pi_K}{(1+\pi_1 +\cdots + \pi_1\cdots\pi_K)^2}$. 
		
		Note that~\eqref{eq:ceq}, as a polynomial in $c$, has only positive coefficients, except for the zero degree coefficient, which is $-1$. 
		By Descartes' rule of sign, a polynomial with coefficients of this form has exactly one real positive root. 
	\end{proof}
	
	Interestingly, the optimal policy of Example~\ref{exopt1} does not assign uniform probabilities to all actions; 
	it satisfies $\pi_1< \cdots < \pi_K$. 
	The interpretation for this is that, when the agent has covered a longer distance toward the reward position $K$, the cost of being transported back to the start position before reaching the reward increases. 
	Hence the policy assigns more probability mass to the `right' actions at positions closer to the reward position. 
	This effect is less pronounced for larger values of $K$, for which the optimal policy is more uniform. 
	For illustration we solved the polynomial~\eqref{eq:ceq} numerically for different values of $K$.  
	The results are shown in Table~\ref{table:unifK}. 
	\begin{table}[h]
		\centering
		\begin{tabular}{c| c|c}
			$K$ & $\pi =(\pi_1,\ldots,\pi_K)$ & $\Rcal = p_K$ \\
			\hline 
			$1$ & $(1)$ & $0.5$\\
			$2$ & $(0.4142, 0.5858)$ & $0.1464$\\%(\sqrt{2}-1, \sqrt{2}-1 + (\sqrt{2}-1)^2)\\
			$3$ & $(0.2744 ,   0.3496 ,   0.3760)$ & $0.0256$\\
			$4$ & $(0.2104,    0.2547,    0.2659,    0.2689)$ &$0.0030$
		\end{tabular}
		\caption{Optimal memoryless policies obtained from Proposition~\ref{proposition:optimal1} for the POMDP Example~\ref{exopt1} for $K=1,\ldots, 4$. }
		\label{table:unifK} 
	\end{table}
\end{example}

\begin{example} 
\label{exopt2}
We consider a slight generalization of Example~\ref{exopt1}. 
Instead of fully deterministic transitions $p(w'|w,a)$ we now assume that at each $w=0,\ldots, K-1$ action $a={w+1}$ takes the agent to $w'=w+1$ with probability $t_{w+1}\in(0, 1]$ and to $w'=0$ with probability $1-t_{w+1}$. 
The world state transition matrix is given by 
	\begin{equation*}
	(p(w'|w))_{w,w'} = 
	\begin{bmatrix}
	1 - t_1\pi_1  & t_1\pi_1 & &&&\\
	1 - t_2\pi_2 && t_2\pi_2  &&&\\
	1 - t_3\pi_3 &&& t_3\pi_3 && \\
	\vdots    &&&&\ddots&\\
	1 - t_K\pi_K &&&&& t_K\pi_K  \\
	1             &&&&&
	\end{bmatrix}. 
	\end{equation*}
For this system the expected reward per time step is $\Rcal(\pi) = p^\pi(w=K) = p_K$. 
Similar to~\eqref{equation:pK1} we find that 
	\begin{equation*}
	p_K = \frac{t_1\pi_1\cdots t_K\pi_K}{ 1 + t_1\pi_1 + \cdots + t_1\pi_1\cdots t_K\pi_K}. 
	\label{eq:eqmaxrewt}
	\end{equation*}
In close analogy to Proposition~\ref{proposition:optimal1} we obtain the following description of the optimal policies. 
\begin{proposition}
\label{proposition:optimal2}
The optimal memoryless policy of the POMDP Example~\ref{exopt2} is given by 
\begin{eqnarray*}
			\pi_1 &=& c\\
			\pi_i &=& \pi_{i-1} + c t_1\pi_1\cdots t_{i-1}\pi_{i-1},\quad i=2,\ldots, K, 
\end{eqnarray*} 
where $c$ is the unique real positive solution of
\begin{equation*} 
		\pi_1+\cdots +\pi_K =1. 
\end{equation*} 
\end{proposition}

The next proposition shows that, in general, for this type of examples, the optimal policy cannot be written as a small convex combination of deterministic policies. 
\begin{proposition}
\label{proposition:bigmix}
There is a choice of $t_1,\ldots,t_K$ for which the optimal memoryless policy of the POMDP Example~\ref{exopt2} cannot be written as a convex combination of $K-1$ deterministic policies. 
\end{proposition}

\begin{proof}[Proof of Proposition~\ref{proposition:bigmix}]
		We show that the set of optimal memoryless policies described in Proposition~\ref{proposition:optimal2}, 
		for all $t_1,\ldots,t_K$, is not contained in any finite union of $K-2$ dimensional affine spaces.  
		
		Consider the expression $\pi_1+\cdots+\pi_K$, where $\pi_1=c$ and $\pi_i = \pi_{i-1} + c t_1\pi_1\cdots t_{i-1}\pi_{i-1}$. 
		We view this as a polynomial in $c$ with coefficients depending on $t_1,\ldots,t_{K}$. 
		The derivative with respect to $t_{K-1}$ is non-zero (as soon as $K\geq 2$ and $c\neq 0$). 
		Hence the solution of $\pi_1+\cdots+\pi_K = 1$ is a non-constant function $c$ of $t_{K-1}$. 
		
		Consider the set of optimal policies for a fixed choice of $t_1,\ldots, t_{K-2}$ and an interval $T\subseteq(0,1]$ of values of $t_{K-1}$. 
		This is given by 
		\begin{equation*}
		(f_1(c) , f_2(c),\ldots, f_{2^{K-1}}(c), f_{2^K}(c)t_{K-1}),\quad\text{for all} t_{K-1}\in T, 
		\end{equation*} 
		where $c$ is a non-constant function of $t_{K-1}$, and 
		$f_j$ is a polynomial of degree $j$ in $c$ with coefficients depending on $t_1,\ldots, t_{K-2}$. 
The restriction of these vectors to the first $K-1$ coordinates is 
		\begin{equation*}
		(f_1(c) , f_2(c),\ldots, f_{2^{K-1}}(c)),\quad\text{for all} c\in C,
		\end{equation*}
		where $C = \{c(t_{K-1}) \colon t_{K-1}\in T\}\subset \mathbb{R}$ is an interval with non-empty interior. 
		This set is a linear projection of the interval $\{ (c, c^2, \ldots, c^{2^K}) \colon c\in C\}$ of the moment curve in $2^K$-dimensional Euclidean space, by the matrix $M$ with entry $M_{j,i}$ equal to the degree-$i$ coefficient of $f_j$, for all $j=1,\ldots,K-1$ and $i=1,\ldots, 2^K$. 
		This matrix is full rank, since each of the $f_j$ has different degree. 
		
		It is well known that each hyperplane intersects a moment curve at most at finitely many points. 
		%This implies that the ambient space is the smallest affine space containing an interval of a moment curve. 
		Since our linear projection is full rank, the smallest affine space containing infinitely many of its points is equal to the ambient space $\mathbb{R}^{K-1}$. 
		In turn, no finite union of convex hulls of $K-1$ polices contains the set of optimal policies for all $t_{K-1}\in T$. 
\end{proof}
\end{example}

\begin{example} 
\label{exopt3}
Consider a POMDP where the agent has $U$ sensor states, $K$ possible actions, and the world state transitions are as shown in Figure~\ref{fig:chain2}. 
The world state transition matrix is given by 
			\begin{equation*}
			(p(w'|w))_{w,w'} = 
			\begin{bmatrix}
			1- t_{11}\pi_{11} &\! t_{11}\pi_{11}\! &&&&& \\
			\vdots    &&\!\ddots\!&&&&\\
			1-t_{1K}\pi_{1K} &&& \!t_{1K}\pi_{1K}\!  &&\\
			%                     &&&& 1 && &&\\
			1-t_{21}\pi_{21} &&&& \!t_{21}\pi_{21}\! &\\
			\vdots &&&&&  \!\ddots\!& \\
			1-t_{UK}\pi_{UK} &&&& & & \!t_{UK}\pi_{UK}\!\\
			1 &&&&&
			\end{bmatrix},  
			\end{equation*}

	\begin{figure}[t]
		\centering
		\vspace{-1cm}
		\begin{tikzpicture}[->,>=stealth',shorten >=1pt,auto,node distance=3cm, thick, main node/.style={circle,fill=blue!20,draw},every loop/.style={<-,shorten <=1pt}]		
		
		\tikzstyle{neu} = [circle, line width=1pt, draw=black, inner sep=.1cm, minimum size = .7cm, fill=bl!5];
		\tikzstyle{neua} = [circle, double, line width=.8pt, draw=black, inner sep=.1cm, minimum size = .7cm, fill=bl!5];
		\tikzstyle{neub} = [circle, line width=2pt, draw=black, inner sep=.1cm, minimum size = .7cm, fill=bl!5];
		\tikzstyle{neuc} = [circle, line width=1pt, draw=black, inner sep=.1cm, minimum size = .7cm, fill=bl!5];
				
		\node[neu] (0) at (0,0) {$11$}; 	
		\node[neu] (1) at (2,0) {$12$}; 	
		\node[neu] (2) at (4,0) {$13$}; 	
		\node[neua] (3) at (8,0) {$1K$}; 	
		\node[rectangle, draw=white, inner sep=0,fill=white] (R) at (6,0) {$\;\cdots\;$};
				
		\path[every node]
		(0) edge node [below] {$\pi_{12}$} (1);
		\path[every node]
		(1) edge node [below] {$\pi_{13}$} (2);
		\path[every node]
		(2) edge node [below] {$\pi_{14}$} (R)
		(R) edge node [below] {$\pi_{1K}$} (3);
				
		\node[neua] (01) at (0,-2) {$21$}; 	
		\node[neua] (11) at (2,-2) {$22$}; 	
		\node[neua] (21) at (4,-2) {$23$}; 	
		\node[neub] (31) at (8,-2) {$2K$}; 	
		\node[rectangle, draw=white, inner sep=0,fill=white] (R1) at (6,-2) {$\;\cdots\;$};
								
		\draw[->] (3) to [out=-90,in=90] (01) node[draw=none,fill=none, above=.5cm] {$\pi_{21}\quad$};
		
		\path[every node]
		(01) edge node [below] {$\pi_{22}$} (11)
		(11) edge node [below] {$\pi_{23}$} (21)
		(21) edge node [below] {$\pi_{24}$} (R1)
		(R1) edge node [below] {$\pi_{2K}$} (31);
		
		\node[neub] (02) at (0,-4) {$31$}; 	
		\node[neub] (12) at (2,-4) {$32$}; 	
		\node[neub] (22) at (4,-4) {$33$}; 	
		\node[neuc] (32) at (8,-4) {$3K$}; 	
		\node[rectangle, draw=white, inner sep=0,fill=white] (R2) at (6,-4) {$\,\cdots$};		
		\node[neu] (St) at (-1,-1) {$10$}; 	
		
		\draw[->] (31) to [out=-90,in=90] (02) node[draw=none,fill=none, above=.5cm] {$\pi_{31}\quad$};
		
		\path[every node]
		(02) edge node [below] {$\pi_{33}$} (12)
		(12) edge node [below] {$\pi_{33}$} (22)
		(22) edge node [below] {$\pi_{34}$} (R2)
		(R2) edge node [below] {$\pi_{3K}$} (32);
				
		\draw[->] (32) to [out=-90,in=0] (3,-5) to [out=180,in=-45] (-.5,-4.5) to [out=135,in=-90] (St);
		\draw[->] (St) to [out=90, in=180]  (0) node[draw=none,fill=none, left] {$\pi_{11}$\quad\quad}; 
		\end{tikzpicture} 
		\vspace{-1cm}
		\caption{State transitions from Example~\ref{exopt3}. 
		The number in each node indicates the world state. 
		The type of circle indicates the sensation of the agent (single stroke stands for $s=1$, double stroke for $s=2$, etc.). 
		At each world state exactly one action takes the agent further, while all other actions take it back to $w=0$ (arrows omitted for clarity). 
		At $w_{UK}$ the agent receives a reward of one and is invariably taken to $w_{10}$. % for $j=1,\ldots, U-1$. 
 		}
		\label{fig:chain2}
	\end{figure}
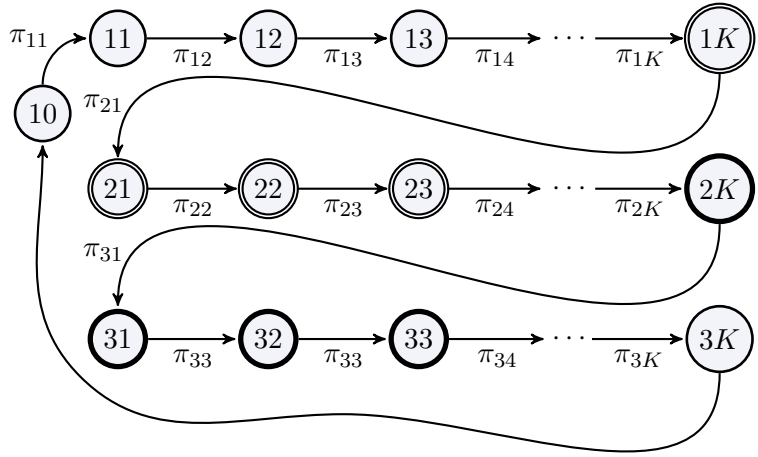	

	The next proposition describes the optimal memoryless policy. 
	\begin{proposition}
		\label{proposition:optimal3}
		The optimal memoryless policy is given by 
		\begin{eqnarray*}
			\pi_{j1} &=& c_j d_j, \\
			\pi_{ji} &=& \pi_{j,i-1} + c_je_j t_{j1}\pi_{j1}\cdots t_{j,i-1}\pi_{j,i-1}, \quad i=2,\ldots, K, 
		\end{eqnarray*}
		where 
		\begin{eqnarray*}
			d_j&=& 1+ t_{11}\pi_{11}+\cdots + t_{11}\pi_{11}\cdots t_{j-1,K}\pi_{j-1,K}, \\
			e_j &=& t_{11}\pi_{11}\cdots t_{j-1,K}\pi_{j-1,K}, 
		\end{eqnarray*}
		and $c_j$ is the unique real positive solution of 
		\begin{equation*}
		\pi_{j1} + \cdots + \pi_{jK} = 1, 
		\end{equation*} 
		for $j=1,\ldots, U$. 
		Here empty products are defined as $1$ and empty sums as $0$. 
	\end{proposition}
	
	\begin{proof}[Proof of Proposition~\ref{proposition:optimal3}]
	After some algebra, similar to~\eqref{equation:pK1}, one finds that the last entry of the stationary world state distribution is given by 
		\begin{equation*}
		p_{UK} = \frac{ t_{11}\pi_{11}\cdots t_{UK}\pi_{UK} }{1+t_{11}\pi_{11}+\cdots + t_{11}\pi_{11}\cdots t_{UK}\pi_{UK}}. 
		\end{equation*}
		We can maximize this with respect to $\pi$ using the method of Lagrange multipliers. 
		This yields the following conditions: 
		\begin{eqnarray*}
			1-\sum_i\pi_{ji} &=& 0,\quad\text{for all $j=1,\ldots,U$},\\
			\frac{\partial p_{UK}}{\partial \pi_{ji}} -\lambda_{j} &=& 0, \quad\text{for all $i=1,\ldots,K$ and $j=1,\ldots,U$}. 
		\end{eqnarray*}
		From this we obtain
		\begin{multline*}
		\lambda_j = 
		\frac{1}{\pi_{ji}}
		\frac{t_{11}\pi_{11}\cdots t_{UK}\pi_{UK}}{(1+t_{11}\pi_{11} + \cdots+t_{11}\pi_{11}\cdots t_{UK}\pi_{UK})} \\
		\times\left(1 - \frac{t_{11}\pi_{11}\cdots t_{ji}\pi_{ji} + \cdots +t_{11}\pi_{11}\cdots t_{UK}\pi_{UK}}{1+t_{11}\pi_{11} + \cdots +t_{11}\pi_{11}\cdots t_{UK}\pi_{UK}}\right),\\ \quad\text{for all  $i=1,\ldots, K$ and $j=1,\ldots, U$}.  
		\end{multline*}
		This implies 
		\begin{equation*}
		\pi_{ji} =  c_j( 1+t_{11}\pi_{11}+\cdots +t_{11}\pi_{11}\cdots t_{j,i-1}\pi_{j,i-1}), 
		\quad\text{for all  $i=1,\ldots, K$ and $j=1,\ldots, U$}, 
		\end{equation*}
		where $c_j = \lambda_j^{-1}\frac{t_{11}\pi_{11}\cdots t_{UK}\pi_{UK}}{(1+t_{11}\pi_{11} +\cdots + t_{11}\pi_{11}\cdots t_{UK}\pi_{UK})^2}$. 
	\end{proof}
	
For each sensor state the optimal policy has $K$ positive probability actions. 
In particular, the smallest face of $\Delta_{S,A}$ which contains the optimal policy has dimension $|U|(|A|-1)$. 	
\end{example}

\end{document}